\documentclass[11pt]{amsart}
\usepackage[english]{babel}
\usepackage[T1]{fontenc}

\usepackage{float}
\usepackage{amsmath}
\usepackage{amsthm}
\usepackage{amssymb}
\usepackage{mathrsfs}
\usepackage{dsfont}
\usepackage{bbm}
\usepackage{enumerate}
\usepackage{graphicx}
\usepackage{caption}
\usepackage{subcaption}
\usepackage{xcolor}
\usepackage{hyperref}

\usepackage{accents}
\usepackage{mathtools}

\newtheorem{theorem}{Theorem}[section]

\newtheorem{lemma}[theorem]{Lemma}
\newtheorem{proposition}[theorem]{Proposition}
\newtheorem{corollary}[theorem]{Corollary}
\newtheorem{definition}[theorem]{Definition}
\theoremstyle{definition}
\newtheorem{remark}[theorem]{Remark}

\newtheorem{example}[theorem]{Example}

\newcommand{\dist}{\operatorname{dist}}
\newcommand{\mad}{\operatorname{MAD}}

\title[$n$-adic Martingale Isoperimetry]{Sharp Ternary Martingale Isoperimetry and $n$-adic Takagi-Type Lower Bounds
}

\author[N. Alpay]{Natanael Alpay}
\address{(NA)
	Department of Mathematics\\ 
	University of California, Irvine,
	Irvine, CA 92697 \\
	USA}
\email{nalpay@uci.edu}
\date{}

\begin{document}

\begin{abstract}
Let $S_1$ be the one-variation associated with the regular
$n$-adic martingale filtration on $[0,1)$.  We study the martingale
isoperimetric profile
\[
V_n(x):=
\inf_{\substack{A\subset[0,1)\ {\rm measurable}\\ |A|=x}}
\|S_1(\mathbbm 1_A)\|_1 .
\]
For the ternary filtration we determine this profile exactly.  Namely,
\[
V_3(x)=T_3(x):=
\sum_{j=0}^{\infty}3^{-j}\psi_3(\{3^j x\}),
\]
where
\[
\psi_3(t)=
\min\left\{
\frac{1+2\left|t-\frac12\right|}{3},
\frac{2-4\left|t-\frac12\right|}{3}
\right\},
\qquad 0\le t\le1 .
\]
Thus the sharp ternary profile is a Takagi-type Bellman function.  It is,
however, not the usual ternary Takagi--van der Waerden function
$\omega_3$; for example,
\[
T_3(1/3)=4/9,
\qquad
\omega_3(1/3)=1/3 .
\]
For general $n\ge2$, we prove that every measurable
$A\subset[0,1)$ satisfies
\[
\|S_1(\mathbbm 1_A)\|_1
\ge
\omega_n(|A|^*)
\asymp_n
|A|^*\log\frac1{|A|^*},
\qquad
|A|^*:=\min\{|A|,1-|A|\}.
\]
Moreover, this logarithmic order is sharp up to a constant depending
only on $n$.

Finally, for every $0<\alpha<1$, we prove the endpoint estimate
\[
\|S_1(\mathbbm 1_A)\|_\alpha\ge |A|^*,
\]
and show that it is sharp up to a constant depending only on $\alpha$
and $n$.
\end{abstract}
\maketitle

\noindent\text{2020 AMS Classification: }{Primary 60G42; Secondary 05C35, 05D05, 42B25, 39B62, 11A63}.\\

\noindent\text{Keywords: }{martingale isoperimetry; $n$-adic filtrations; Bellman functions; Takagi functions; edge-isoperimetric inequalities; digit sums; Boolean functions.}

\section{Introduction}\label{into}

Isoperimetric inequalities study how small the boundary of a set can be when
only its size is prescribed.  In discrete mathematics this question often
depends on the way mass is arranged across a finite or recursive structure.
On the Boolean cube, the basic example is the edge boundary: for
$A\subset\{-1,1\}^m$, the $L^1$ norm of the Boolean gradient of
$\mathbbm 1_A$ counts, up to normalization, the edges joining $A$ to its
complement.  Sharp forms of this problem are closely tied to digit structure.
For instance, Hart's formula expresses the extremal number of internal edges
in terms of cumulative binary digit sums \cite{hart,hart1976note}.  In a
related direction, Lev studied edge-isoperimetric problems through generalized
Takagi functions \cite{LevEdgeIso}.  These results are one of the main
motivations for the present paper.

A second point of comparison comes from Talagrand's two-sided gradient
isoperimetric inequality on the Boolean cube.  In one formulation, for
$A\subset\{-1,1\}^m$ and $q\ge 1/2$,
\begin{equation}\label{eq:intro-talagrand}
\bigl\||\nabla \mathbbm 1_A|_1\bigr\|_q
\ge
C (|A|^*)^{1/q}\log\frac1{|A|^*},
\qquad
|A|^*:=\min\{|A|,1-|A|\},
\end{equation}
where $C>0$ is a universal constant; see \cite{DIR,BIM}.  When $q=1$,
the left-hand side is the edge boundary of $A$.  More generally, for
Boolean functions $f=\mathbbm 1_A$, the $\beta$-gradient is determined by
the $1$-gradient, since $|2D_j f|\in\{0,1\}$.  Thus lower bounds for
Boolean gradients reduce, in this setting, to lower bounds for the
one-gradient.  This is the Boolean-cube motivation for studying the
one-variation in the $n$-adic martingale setting. For more background and a deeper look at the relation between the hypercube and the dyadic interval see {\cite[Section 1.1]{alpay2025lower}}.

The object studied here is the rooted regular $n$-ary tree, or equivalently
the regular $n$-adic filtration on $[0,1)$.  For the purposes of the
introduction, let $f=\mathbbm 1_A$, let $f_m$ be the conditional
expectation of $f$ onto the level $m$ $n$-adic partition, and put
$d_m=f_m-f_{m-1}$.  We write
\[
S_1(f)(y):=\sum_{m\ge1}|d_m(y)|,
\]
with the value $+\infty$ allowed.  This is the one-variation associated
with the filtration.

If $I$ is an $n$-adic interval, write
\[
x_I:=\frac{|A\cap I|}{|I|}
\]
for the density of $A$ on $I$.  If $I_1,\ldots,I_n$ are the children of
$I$, then on the child $I_i$ the next martingale difference is
\[
x_{I_i}-x_I .
\]
Thus the contribution of this splitting to the $L^1$ one-variation is
\[
|I|\,\frac1n\sum_{i=1}^n |x_{I_i}-x_I|.
\]
The local boundary cost is therefore the mean absolute deviation of the child
densities from their parent density.  In this sense
$\|S_1(\mathbbm 1_A)\|_1$ is a natural boundary energy on the rooted
$n$-ary tree.

The main profile considered in this paper is
\[
V_n(x):=
\inf_{\substack{A\subset[0,1)\ {\rm measurable}\\ |A|=x}}
\|S_1(\mathbbm 1_A)\|_1,
\qquad 0\le x\le1.
\]
The problem is recursive: at each vertex one splits a parent density
$\bar x$ into child densities $x_1,\ldots,x_n$, and pays the local cost
\[
\frac1n\sum_{i=1}^n |x_i-\bar x|,
\qquad
\bar x=\frac1n\sum_{i=1}^n x_i.
\]
We formalize this recursion using Bellman functions.  The Bellman principle
proved below shows that every admissible $n$-point Bellman function gives a
global lower bound for $\|S_\beta(\mathbbm 1_A)\|_\alpha$.  This is the
basic framework of the paper.

For $n=2$, the recursion leads to the classical Takagi function (studied in \cite{alpay2025lower}).  The first
new phenomenon occurs in base $3$.  One might expect the sharp ternary
profile to be the usual ternary Takagi--van der Waerden function, but it is
not.  The reason is that the local cost here is mean absolute deviation, not
the range of the child densities (as in the work by Lev \cite{LevEdgeIso}).  The optimal one-step ternary splitting has
at most one residual child and leads to the generator
\[
\psi_3(t)=
\min\left\{
\frac{1+2\left|t-\frac12\right|}{3},
\frac{2-4\left|t-\frac12\right|}{3}
\right\},
\qquad 0\le t\le1.
\]
Our main sharp result is
\[
V_3(x)=T_3(x):=
\sum_{j=0}^{\infty}3^{-j}\psi_3(\{3^j x\}),
\qquad 0\le x\le1,
\]
where $\{t\}=t-\lfloor t\rfloor$ denotes the fractional part of $t$.
This is a Takagi-type Bellman function, but it is not the usual ternary
Takagi--van der Waerden function $\omega_3$.  For example,
\begin{equation}
\label{eq:TvsW}
    T_3\left(\frac13\right)=\psi_3\left(\frac13\right)=\frac49,
\qquad
\omega_3\left(\frac13\right)=\frac13.
\end{equation}
Thus the sharp ternary profile belongs to the same circle of ideas as digit
sums, Takagi functions, and edge-isoperimetry, but it is a genuinely different
profile.

For general $n$, the sharp profile appears to be more complicated.  Already
for $n=4$, the naive one-residual analogue of the ternary formula fails the
Bellman inequality (see Remark \ref{rmk:ngeq4}).  We therefore prove a general lower bound rather than a
sharp formula in all bases.  Let $\sigma_n(k)$ be the sum of the base-$n$
digits of $k$, and let
\[
G_n(M):=\sum_{k=0}^{M-1}\sigma_n(k).
\]
Using a summatory digit-sum inequality of Allouche and Stipulanti
\cite{AlloucheStipulanti2025Summing}, we construct an admissible Bellman
function $P^{(n)}$.  It has a Takagi-type representation
\[
P^{(n)}(x)=
\sum_{j=0}^{\infty} n^{-j}\eta_n(\{n^j x\}),
\qquad
\eta_n\left(\frac{\rho}{n}\right)
=
\frac{\rho(n-\rho)}{n(n-1)} .
\]
This function dominates the usual $n$-adic Takagi--van der Waerden function
\[
P^{(n)}(x)\ge \omega_n(x),
\qquad 0\le x\le1.
\]
Consequently, for every measurable set $A\subset[0,1)$,
\[
\|S_1(\mathbbm 1_A)\|_1
\ge
P^{(n)}(|A|)
\ge
\omega_n(|A|^*)
\asymp_n
|A|^*\log\frac1{|A|^*}.
\]
For $n=2$ and $n=3$, this construction gives
$P^{(n)}=\omega_n$.  For $n\ge4$, the comparison is generally strict.
The logarithmic order of this lower bound is optimal. More precisely,
for the leftmost $n$-adic intervals
$
A_k=[0,n^{-k}),$ such that $
p_k=n^{-k},
$
we have
\[
P^{(n)}(p_k)=\omega_n(p_k)=kp_k
\quad 
\text{and}
\quad
\|S_1(\mathbbm 1_{A_k})\|_1
=
\frac{2(n-1)}{n}\,kp_k.
\]
Since
\[
kp_k
=
p_k\log_n\frac1{p_k},
\]
this shows that the general lower bound has the correct asymptotic
order as $p_k\to0$. For $n\ge4$, this does not identify the exact
profile; it only proves sharpness up to a constant.

Finally, we consider the quasi-norm range $0<\alpha<1$.  In that range the
logarithmic factor disappears: the boundary-distance function
$x^*=\min\{x,1-x\}$ is itself admissible, and hence
\[
\|S_1(\mathbbm 1_A)\|_\alpha\ge |A|^*.
\]
The leftmost $n$-adic intervals $[0,n^{-k})$ show that this estimate has
the correct order as $|A|\to0$, up to a constant depending only on $n$ and
$\alpha$.

The paper is organized as follows.  Section~2 states the main results.
Section~3 proves the Bellman principle.  Section~4 proves the sharp ternary
profile.  Sections~5 and~6 develop the digit-sum Bellman function and compare
it with $\omega_n$.  Section~7 proves the sub-$L^1$ endpoint theorem.  The
appendix contains the exact verification used in the ternary compression
argument.

\subsection{Framework}
Fix an integer $n\ge2$.  Let $(\Omega,\mathcal L,dx)$ be $[0,1)$ with
Lebesgue measure.  For $m\ge0$, let $\mathcal I_m=\mathcal I_m^{(n)}$ be
the collection of $n$-adic intervals of level $m$,
\[
\mathcal I_m=
\left\{
\left[\frac{k}{n^m},\frac{k+1}{n^m}\right):
k=0,\ldots,n^m-1
\right\}.
\]
Set
\[
\mathcal N_0^{(n)}\subset\mathcal N_1^{(n)}\subset\cdots,
\qquad
\mathcal N_m^{(n)}:=\sigma(\mathcal I_m),
\qquad
\mathcal N:=\bigcup_{m\ge0}\mathcal N_m^{(n)}.
\]
We also write
\[
N_m=N_m^{(n)}:=\left\{\frac{k}{n^m}:k=0,\ldots,n^m\right\},
\qquad
N:=\bigcup_{m\ge0}N_m.
\]
For $f\in L^1[0,1)$ and an interval $I\subset[0,1)$, put
\[
\langle f\rangle_I:=\frac1{|I|}\int_I f(y)\,dy.
\]
The $n$-adic martingale generated by $f$ is
\[
f_m=\mathbb E(f\mid\mathcal N_m)
=\sum_{I\in\mathcal I_m}\mathbbm 1_I\langle f\rangle_I,
\qquad m\ge0.
\]
Let $d_m=f_m-f_{m-1}$ for $m\ge1$.  For $\beta\ge1$, define
\[
S_{\beta,M}(f)=\left(\sum_{m=1}^M |d_m|^\beta\right)^{1/\beta},
\qquad
S_\beta(f)=\lim_{M\to\infty}S_{\beta,M}(f),
\]
with the value $+\infty$ allowed.  When $\beta=1$, this is the one
variation.

The classical Burkholder--Davis--Gundy inequality compares $S_2(f)$ with
$f-\mathbb Ef$ in $L^p$, $1<p<\infty$ \cite{BDS1,BDS2}.  For dyadic
Paley--Walsh martingales one also has conditional symmetry and sharper endpoint
estimates \cite{Wang1,Wang2,Davis}.  For $n\ge3$, a general $n$-adic
martingale generated by an indicator function is not conditionally symmetric:
for instance $A=[0,1/n)$ gives a first martingale difference taking the two
values $1-1/n$ and $-1/n$.

\subsection{Takagi-type functions}
The classical Takagi--Knopp function is
\[
T(x)=\sum_{k=0}^{\infty}2^{-k}\operatorname{dist}(2^kx,\mathbb Z),
\qquad x\in\mathbb R.
\]
It was introduced by Takagi and Knopp \cite{Takagi1903,Knopp1918}, and it has
many connections with digit sums, fractal geometry, and isoperimetric problems
\cite{BorosPales2001,AllaartKawamura2011,allaart2014hausdorff,LevEdgeIso}.
The functions in this paper have the same self-similar form, but with the base
and the one-step generator dictated by the martingale Bellman problem:
\[
T_{\psi,n}(x)=\sum_{k=0}^{\infty}n^{-k}\psi(\{n^kx\}),
\qquad 0\le x\le1.
\]
Here $\psi$ is continuous, non-negative, and vanishes at $0$ and $1$, so
the series converges uniformly.

The usual $n$-adic Takagi--van der Waerden function is
\begin{equation}
\label{eq:wn}
\omega_n(x)=\sum_{k=0}^{\infty}\frac{\|n^kx\|_{1/n}}{n^k},
\qquad
\|t\|_{1/n}:=\min\{\operatorname{dist}(t,\mathbb Z),1/n\}.
\end{equation}
We shall use the following standard facts.

\begin{lemma}\label{thm:wn-properties}
For every $n\ge2$, the function $\omega_n$ has the following properties.
\begin{enumerate}[(a)]
\item The defining series converges uniformly, and hence $\omega_n$ is
continuous on $[0,1]$ \cite[Section 7.2]{AllaartKawamura2011}.
\item If $m=tn+\rho$, where $1\le\rho\le n-1$, then for every integer
$r\ge1$,
\[
\omega_n\left(\frac{m}{n^r}\right)
=\left(1-\frac{\rho}{n}\right)\omega_n\left(\frac{t}{n^{r-1}}\right)
+\frac{\rho}{n}\omega_n\left(\frac{t+1}{n^{r-1}}\right)
+\frac1{n^r}
\]
\cite[Lemma 1]{LevEdgeIso}.
\item $\omega_n(1-x)=\omega_n(x)$ for all $x\in[0,1]$.
\item For $0<x<1$,
\[
\omega_n(x)\asymp_n x^*\log\frac1{x^*},
\qquad x^*:=\min\{x,1-x\}.
\]
\end{enumerate}
\end{lemma}

\begin{figure}[ht]
\centering
\includegraphics[width=0.8\linewidth]{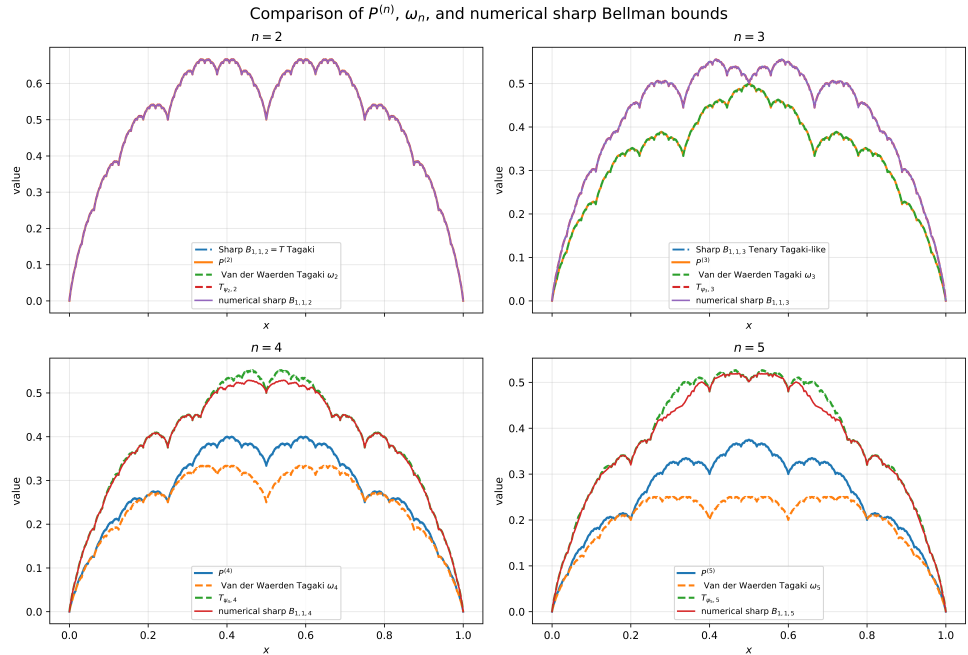}
\caption{Comparison of several Takagi-type candidates on $n$-adic grids.  For
$n=2$ the natural candidates coincide.  For $n=3$, the sharp ternary
Bellman function $T_3$ is strictly larger than the usual
Takagi--van der Waerden function $\omega_3$ at some points.  For
$n=4,5$, the picture shows the lower bound $\omega_n$, together with
experimental Bellman profiles and the naive one-residual extension; the latter
fails to be admissible for $n\ge4$.}
\label{fig:nadic}
\end{figure}

\section{Main Results}

Fix parameters $0<\alpha\le1\le\beta$ and an integer $n\ge2$.

\begin{definition}\label{npi}
Let $\mathcal B(\alpha,\beta,n)$ be the family of all continuous
non-negative functions $F:[0,1]\to[0,\infty)$ satisfying
$F(0)=F(1)=0$ and the following $n$-point inequality: for every
$x_1,\ldots,x_n\in[0,1]$,
\begin{equation}\label{eq:npt}
F(\bar x)^\alpha
\le
\frac1n\sum_{i=1}^n
\left(F(x_i)^\beta+|x_i-\bar x|^\beta\right)^{\alpha/\beta},
\qquad
\bar x=\frac1n\sum_{i=1}^n x_i.
\end{equation}
\end{definition}

The first result is the Bellman principle which turns the finite-dimensional
inequality \eqref{eq:npt} into a lower bound for all measurable sets.

\begin{theorem}[Bellman bound]\label{mthm:1}
Let $n\ge2$ and $0<\alpha\le1\le\beta$.  Then for every
$F\in\mathcal B(\alpha,\beta,n)$ and every measurable set
$A\subset[0,1)$,
\[
\|S_\beta(\mathbbm 1_A)\|_\alpha\ge F(|A|).
\]
\end{theorem}

When needed, we write
\[
B_{\alpha,\beta,n}(x):=\sup\{F(x):F\in\mathcal B(\alpha,\beta,n)\}.
\]
No continuity or admissibility of this pointwise envelope is needed in the sequel.
For completeness, one may also construct a maximal continuous admissible
representative by the same Perron-envelope argument as in
\cite{alpay2025lower}, adapted to the present $n$-point inequality; we do
not use this additional fact here.

The main theorem of the paper is the exact solution in the ternary case.

\begin{theorem}[Sharp ternary profile]\label{mthm:2}
For $n=3$ and $\alpha=\beta=1$, for every $0\le x\le1$,
\[
\inf_{\substack{A\subset[0,1)\ {\rm measurable}\\ |A|=x}}
\|S_1(\mathbbm 1_A)\|_1
=
T_3(x):=\sum_{j=0}^{\infty}3^{-j}\psi_3(\{3^jx\}),
\]
where $\{t\}=t-\lfloor t\rfloor$ denotes the fractional part of $t$, and
\[
\psi_3(t)=
\min\left\{
\frac{1+2\left|t-\frac12\right|}{3},
\frac{2-4\left|t-\frac12\right|}{3}
\right\},
\qquad 0\le t\le1.
\]
Moreover $B_{1,1,3}=T_3$, and for $0<x<1$,
\[
T_3(x)\asymp x^*\log\frac1{x^*},
\qquad x^*=\min\{x,1-x\}.
\]
\end{theorem}

\begin{remark}
\label{rmk:sharpness-omega3}
The lower bound in Theorem \ref{mthm:2} is not the usual ternary
Takagi--van der Waerden function, see \eqref{eq:TvsW}.
Thus $\omega_3$ has the correct logarithmic order, but it is not the sharp
Bellman function for the mean-deviation one-variation problem.
\end{remark}

\begin{remark}
\label{rmk:ngeq4}
For $n\ge 4$, the same Takagi-type viewpoint remains useful, but the sharp profile is more complicated. A natural first attempt is to use the one-residual generator
\[
\psi_n(t)
=
\begin{cases}
\dfrac{2k(1-t)}{n},
&
\dfrac{k}{n}\le t\le \dfrac{k}{n-1},
\\[2ex]
\dfrac{2(n-k-1)t}{n},
&
\dfrac{k}{n-1}\le t\le \dfrac{k+1}{n}
\end{cases},
\]
for $k=0,1,\dots,n-1$, on the interval
$
\tfrac{k}{n}\le t\le \tfrac{k+1}{n}$. 
However, this naive $n$-adic analogue does not give an admissible lower-bound Bellman function for $S_1$. See Figure \ref{fig:nadic}.
 
To see the obstruction, take $n=4$ and 
\[
T_{\rm naive}(x):=\sum_{j=0}^{\infty}4^{-j}\psi_4(\{4^jx\}),
\]
where $\psi_4$ is the one-residual generator above.
Consider the combination
\[
(x_1,x_2,x_3,x_4)=\left(0,\frac14,\frac5{16},1\right),
\qquad
\bar x=\frac{25}{64}.
\]
The left-hand side of the Bellman inequality \eqref{eq:npt} is
\[
T_{\rm naive}(\bar x)=\frac{67}{128},
\]
whereas the right-hand side is
\[
\frac14\sum_{i=1}^4T_{\rm naive}(x_i)
+
\frac14\sum_{i=1}^4|x_i-\bar x|
=
\frac{65}{128}.
\]
Thus $T_{\text{naive}}$ fails the Bellman inequality \eqref{eq:npt} by $1/64$. This is why the general
$n$-adic theorem below is stated as a robust lower bound rather than as the
sharp profile.
\end{remark}

For general $n$, the following theorem gives the main lower bound.

\begin{theorem}[Takagi--van der Waerden lower bound]\label{mthm4}
Let $n\ge2$, $\alpha=\beta=1$, and let $A\subset[0,1)$ be measurable
with $|A|=x$.  Then
\[
\|S_1(\mathbbm 1_A)\|_1
\ge
\omega_n(x^*)
\asymp_n
x^*\log\frac1{x^*},
\qquad
x^*=\min\{x,1-x\}.
\]
Moreover, the lower bound is sharp up to a constant. That is there exists a constant $c_n$ depending on $n$ such that for $A_k = [0,n^{-k})$, we have
\[
\|S_1(\mathbbm 1_{A_k})\|_1= c_n \omega_n(n^{-k}),\quad k\geq 1.
\]
For $n=2$, this recovers the dyadic Takagi lower bound in the present
normalization.
\end{theorem}

To prove Theorem \ref{mthm4}, we do not try to verify the Bellman inequality
for $\omega_n$ directly, this proves to be a hard task.
Instead we build a larger admissible function from
summatory digit sums.  This is the arithmetic part of the paper, and it is
also the point at which the analogy with Hart's digit-sum formula becomes most
visible.\\

For $k\ge0$, write its base-$n$ expansion as
\[
k=\sum_{j=0}^{m-1}\varepsilon_j(k)n^j,
\qquad \varepsilon_j(k)\in\{0,1,\ldots,n-1\},
\]
and set
\[
\sigma_n(k):=\sum_{j=0}^{\infty}\varepsilon_j(k),
\qquad
G_n(M):=\sum_{k=0}^{M-1}\sigma_n(k),
\qquad G_n(0):=0.
\]
For $x\in N_m$, define
\begin{equation}\label{eq:Bn-m-def}
B_m^{(n)}(x):=
mx-
\frac{2}{(n-1)n^m}G_n(n^mx).
\end{equation}

\begin{theorem}[Digit-sum Bellman function]\label{Thm:digit-sum-Pn}
Using \eqref{eq:Bn-m-def}, define
\[
P^{(n)}(x):=\lim_{m\to\infty}B_m^{(n)}(x),
\qquad x\in N.
\]
Then:
\begin{enumerate}
\item If $x\in N_{m_0}$, then
\[
B_m^{(n)}(x)=B_{m_0}^{(n)}(x)
\qquad \text{for every }m\ge m_0.
\]
In particular the limit defining $P^{(n)}$ exists on $N$.

\item $P^{(n)}(0)=P^{(n)}(1)=0$.

\item For every $x_1,\ldots,x_n\in N$,
\[
P^{(n)}(\bar x)
\le
\frac1n\sum_{i=1}^nP^{(n)}(x_i)
+
\frac1n\sum_{i=1}^n|x_i-\bar x|,
\qquad
\bar x=\frac1n\sum_{i=1}^n x_i.
\]

\item The function $P^{(n)}$ extends continuously to $[0,1]$, still
denoted by $P^{(n)}$, and $P^{(n)}\in\mathcal B(1,1,n)$.  Hence for every
measurable set $A\subset[0,1)$,
\[
P^{(n)}(|A|^*)\le\|S_1(\mathbbm 1_A)\|_1.
\]

\item The continuous extension has the Takagi-type representation
\[
P^{(n)}(x)=\sum_{j=0}^{\infty}n^{-j}\eta_n(\{n^jx\}),
\]
where $\eta_n:[0,1]\to[0,\infty)$ is the continuous piecewise-linear
function determined by
\[
\eta_n\left(\frac{\rho}{n}\right)=\frac{\rho(n-\rho)}{n(n-1)},
\qquad \rho=0,1,\ldots,n.
\]
\item The continuous extension is symmetric:
\[
P^{(n)}(1-x)=P^{(n)}(x),\qquad 0\le x\le1.
\]
\item For $0<x<1$,
\[
P^{(n)}(x)\asymp_n x^*\log\frac1{x^*},
\qquad x^*=\min\{x,1-x\}.
\]
\end{enumerate}
\end{theorem}

The comparison with $\omega_n$ is pointwise at the level of the generators.

\begin{corollary}\label{cor:P-dominates-omega}
For every $n\ge2$ and every $x\in[0,1]$,
\[
P^{(n)}(x)\ge \omega_n(x).
\]
where equality holds at $x=n^{-k}$ for $k\geq 1$.
Moreover,
\[
P^{(2)}=\omega_2,
\qquad
P^{(3)}=\omega_3.
\]
For $n\ge4$, the comparison is generally strict.
\end{corollary}

The last theorem concerns the quasi-norm range $0<\alpha<1$.

\begin{theorem}[The sub-$L^1$ endpoint]\label{thm:alpha}
Let $n\ge2$ and $0<\alpha<1$.  Then for every measurable set
$A\subset[0,1)$,
\[
\|S_1(\mathbbm 1_A)\|_\alpha\ge |A|^*.
\]
Moreover, for each $\alpha\in(0,1)$, there exists a constant
$C_{\alpha,n}<\infty$ and a sequence of sets $A_j\in\mathcal N$ such that
$|A_j|\to0$ and
\[
\|S_1(\mathbbm 1_{A_j})\|_\alpha\le C_{\alpha,n}|A_j|.
\]
\end{theorem}

\subsection*{Open problem}
The natural remaining problem is to determine the sharp Bellman function
$B_{1,1,n}$ for $n\ge4$.  The ternary theorem shows that the correct
answer can be a non-classical Takagi-type function, while Remark
\ref{rmk:ngeq4} shows that the most direct one-residual generalization already
fails in base $4$.  It would be interesting to know whether the sharp
profiles for $n\ge4$ still admit a finite-state Takagi-type description, or
whether a genuinely new phenomenon appears.

\subsection*{Acknowledgements.}
The author would like to thank Professor Paata Ivanisvili for valuable
discussions, insightful suggestions, and guidance.

\section{Proof of Theorem \ref{mthm:1}}

In this section we prove Theorem \ref{mthm:1}. Thus we fix
$0<\alpha\le 1\le \beta$, an integer $n\ge2$, and a function
$F\in\mathcal B(\alpha,\beta,n)$. We shall prove directly that, for every
measurable set $A\subset[0,1)$,
\[
\|S_\beta(\mathbbm 1_A)\|_\alpha\ge F(|A|).
\]
We first present an extension of the $n$-point Bellman
inequality.

\begin{lemma}
\label{lem:common-tail}
Let $0<r\le1$. For every real $X_1,\dots,X_n\geq 0$ and every $s\ge0$,
\[
\left(\frac1n\sum_{i=1}^n (X_i+s)^r\right)^{1/r}
\ge
s+
\left(\frac1n\sum_{i=1}^n X_i^r\right)^{1/r}.
\]
\end{lemma}

\begin{proof}

For $r=1$, this is an equality. Assume $0<r<1$. For $s\ge0$, set
\[
\Phi(s)
:=
\left(\frac1n\sum_{i=1}^n (X_i+s)^r\right)^{1/r}-s.
\]
It is enough to show that $\Phi$ is nondecreasing.
For $s>0$, write $Y_i=X_i+s$ and
\[
A(s):=\frac1n\sum_{i=1}^n Y_i^r.
\]
Differentiating gives
\[
\Phi'(s)
=
A(s)^{1/r-1}
\left(\frac1n\sum_{i=1}^n Y_i^{r-1}\right)
-1.
\]
Thus we need to prove
\[
A(s)^{(1-r)/r}
\left(\frac1n\sum_{i=1}^n Y_i^{r-1}\right)
\ge1.
\]
By H\"older's inequality with conjugate exponents $1/(1-r)$ and $1/r$,
\[
1
=
\frac1n\sum_{i=1}^n
Y_i^{r(1-r)}Y_i^{r(r-1)}
\le
\left(\frac1n\sum_{i=1}^n Y_i^r\right)^{1-r}
\left(\frac1n\sum_{i=1}^n Y_i^{r-1}\right)^r.
\]
This is exactly the desired inequality. Hence $\Phi'(s)\ge0$ for
$s>0$. By continuity at $s=0$, $\Phi(s)\ge\Phi(0)$ for every
$s\ge0$, which proves the lemma.
\end{proof}

\begin{lemma}
\label{lem:Bellman_q_extension}
For
$0<\alpha\le1\le\beta$ and $q\ge0$, define
\[
Q_F(p,q)
:=
\left(F(p)^\beta+q^\beta\right)^{\alpha/\beta},
\qquad 0\le p\le1.
\]
Then, for all $q\ge0$ and all $x_1,\dots,x_n\in[0,1]$, one has
\begin{equation}
\label{eq:Q_npoint}
Q_F(\bar x,q)
\le
\frac1n\sum_{i=1}^n
Q_F\left(
x_i,
\left(q^\beta+|x_i-\bar x|^\beta\right)^{1/\beta}
\right),
\qquad
\bar x=\frac1n\sum_{i=1}^n x_i.
\end{equation}
\end{lemma}

\begin{proof}
Set $r:=\frac{\alpha}{\beta},$ since $0<\alpha\le1\le\beta$, we have $0<r\le1$. For each $i$, put
\[
X_i:=F(x_i)^\beta+|x_i-\bar x|^\beta.
\]
Note that $X_i$ are simply non-negative real
numbers.
The defining $n$-point inequality for $F$ gives
\[
F(\bar x)^\alpha
\le
\frac1n\sum_{i=1}^n X_i^r.
\]
Equivalently,
\[
F(\bar x)^\beta
\le
\left(\frac1n\sum_{i=1}^n X_i^r\right)^{1/r}.
\]
Applying Lemma \ref{lem:common-tail} with $s=q^\beta$, we get
\[
\left(\frac1n\sum_{i=1}^n (X_i+q^\beta)^r\right)^{1/r}
\ge
q^\beta+
\left(\frac1n\sum_{i=1}^n X_i^r\right)^{1/r}
\ge
q^\beta+F(\bar x)^\beta.
\]
Raising both sides to the power $r$ yields
\[
\left(F(\bar x)^\beta+q^\beta\right)^r
\le
\frac1n\sum_{i=1}^n
\left(
F(x_i)^\beta+|x_i-\bar x|^\beta+q^\beta
\right)^r.
\]
This is precisely \eqref{eq:Q_npoint}.
\end{proof}

We now prove the theorem. Let $A\subset[0,1)$ be measurable and set
$
f=\mathbbm 1_A.
$
Let
\[
f_m=\mathbb E(f\mid\mathcal N_m),
\qquad
d_m=f_m-f_{m-1},
\]
be the martingale and its differences, moreover, since $f=\mathbbm 1_A$, each $f_m$ takes values in $[0,1]$. For $M\ge0$, write
\[
S_{\beta,M}(f)
=
\left(\sum_{m=1}^M |d_m|^\beta\right)^{1/\beta},
\]
with the convention $S_{\beta,0}(f)=0$. For $q\ge0$, define
\[
R_M^q(x)
:=
\left(
q^\beta+\sum_{m=1}^M |d_m(x)|^\beta
\right)^{1/\beta}.
\]
Thus $R_0^q\equiv q$, and $R_M^q$ is $\mathcal N_M$-measurable.
We claim that for every $M\ge0$ and every $q\ge0$,
\begin{equation}
\label{eq:finite-level-bellman}
Q_F(|A|,q)
\le
\int_0^1 Q_F(f_M(x),R_M^q(x))\,dx.
\end{equation}
For $M=0$, this is an equality, since $f_0\equiv |A|$ and
$R_0^q\equiv q$.\\

Assume that \eqref{eq:finite-level-bellman} is known at level $M$. We
show that the right-hand side can only increase when passing from $M$ to
$M+1$. Let $I\in\mathcal I_M$, and let $I_1,\dots,I_n$ be its
children. On $I$, both $f_M$ and $R_M^q$ are constant. Write
\[
\bar x:=\langle f\rangle_I,
\qquad
\rho:=R_M^q|_I,
\qquad
x_i:=\langle f\rangle_{I_i},
\quad i=1,\dots,n.
\]
Then
\[
\bar x=\frac1n\sum_{i=1}^n x_i.
\]
Moreover, on the child $I_i$,
\[
f_{M+1}=x_i,
\qquad
d_{M+1}=x_i-\bar x,
\]
and therefore
\[
R_{M+1}^q
=
\left(\rho^\beta+|x_i-\bar x|^\beta\right)^{1/\beta}
\qquad\text{on }I_i.
\]
Applying Lemma \ref{lem:Bellman_q_extension} gives
\[
Q_F(\bar x,\rho)
\le
\frac1n\sum_{i=1}^n
Q_F\left(
x_i,
\left(\rho^\beta+|x_i-\bar x|^\beta\right)^{1/\beta}
\right).
\]
Multiplying by $|I|$, this becomes
\[
\int_I Q_F(f_M,R_M^q)\,dx
\le
\int_I Q_F(f_{M+1},R_{M+1}^q)\,dx.
\]
Summing over all $I\in\mathcal I_M$, we obtain
\[
\int_0^1 Q_F(f_M,R_M^q)\,dx
\le
\int_0^1 Q_F(f_{M+1},R_{M+1}^q)\,dx.
\]
Since equality holds at level $0$, this proves
\eqref{eq:finite-level-bellman} for every $M$.
We now take $q=0$. Since $Q_F(|A|,0)=F(|A|)^\alpha$, the finite-level
inequality gives
\[
F(|A|)^\alpha
\le
\int_0^1
\left(
F(f_M(x))^\beta
+
S_{\beta,M}(f)(x)^\beta
\right)^{\alpha/\beta}
\,dx.
\]
Again put $r=\alpha/\beta$. Since $0<r\le1$, we have 
$
(a+b)^r\le a^r+b^r$ for $a,b\ge0.
$
Hence
\[
F(|A|)^\alpha
\le
\int_0^1 F(f_M(x))^\alpha\,dx
+
\int_0^1 S_{\beta,M}(f)(x)^\alpha\,dx.
\]

We pass to the limit $M\to\infty$. 
The $n$-adic filtration generates the Borel sigma algebra modulo null
sets; hence the martingale convergence theorem, equivalently the
Lebesgue differentiation theorem for $n$-adic intervals, gives
$f_M\to f$ a.e., that is
\[
f_M\to f=\mathbbm 1_A
\qquad\text{a.e.}
\]
Since $F$ is continuous on $[0,1]$ and $F(0)=F(1)=0$, it follows that
\[
F(f_M)\to F(\mathbbm 1_A)=0
\qquad\text{a.e.}
\]
Also $F$ is bounded on $[0,1]$, so dominated convergence gives
\[
\int_0^1 F(f_M)^\alpha\,dx\to0.
\]
On the other hand,
\[
S_{\beta,M}(f)^\beta
=
\sum_{m=1}^M |d_m|^\beta
\uparrow
\sum_{m=1}^\infty |d_m|^\beta
=
S_\beta(f)^\beta,
\]
and therefore
\[
S_{\beta,M}(f)^\alpha
\uparrow
S_\beta(f)^\alpha.
\]
By monotone convergence,
\[
\int_0^1 S_{\beta,M}(f)^\alpha\,dx
\to
\int_0^1 S_\beta(f)^\alpha\,dx,
\]
possibly with value $+\infty$. Consequently,
\[
F(|A|)^\alpha
\le
\int_0^1 S_\beta(\mathbbm 1_A)^\alpha\,dx.
\]
Taking the power $1/\alpha$, with the usual convention if the right-hand
side is infinite, gives
\[
F(|A|)
\le
\left(
\int_0^1 S_\beta(\mathbbm 1_A)^\alpha\,dx
\right)^{1/\alpha}
=
\|S_\beta(\mathbbm 1_A)\|_\alpha.
\]
This proves Theorem \ref{mthm:1}.

\section{Proof of Theorem \ref{mthm:2}}
\label{sec:proof-mthm2}

Throughout this section we fix $n=3$ and $\alpha=\beta=1$. In this case
Definition \ref{npi} says that an admissible Bellman function is a
non-negative continuous function $F:[0,1]\to[0,\infty)$, with
$F(0)=F(1)=0$, satisfying
\begin{equation}\label{eq:ternary-bellman-ineq}
F(\bar x)
\le
\frac13\sum_{i=1}^{3}F(x_i)
+
\frac13\sum_{i=1}^{3}|x_i-\bar x|,
\qquad
\bar x=\frac{x_1+x_2+x_3}{3}.
\end{equation}

For $0\le x\le1$, define
\[
V_3(x)
:=
\inf_{\substack{A\subset[0,1)\ {\rm measurable}\\ |A|=x}}
\|S_1(\mathbbm 1_A)\|_1.
\]
We shall prove that
$$
V_3(x)=T_3(x),
$$
with $T_3$ defined by
\begin{equation}\label{eq:T3-def-proof}
T_3(x)
:=
\sum_{j=0}^{\infty}
3^{-j}\psi_3(\{3^jx\}),
\qquad 0\le x\le1,
\end{equation}
where $\{t\}=t-\lfloor t\rfloor$ denotes the fractional part of $t$, and
\begin{equation}\label{eq:psi3-piecewise-proof}
\psi_3(t)
=
\begin{cases}
\dfrac{4t}{3}, & 0\le t\le \dfrac13,\\[1ex]
\dfrac{2(1-t)}{3}, & \dfrac13\le t\le \dfrac12,\\[1ex]
\dfrac{2t}{3}, & \dfrac12\le t\le \dfrac23,\\[1ex]
\dfrac{4(1-t)}{3}, & \dfrac23\le t\le1.
\end{cases}
\end{equation}
Equivalently,
\[
\psi_3(t)
=
\min\left\{
\frac{1+2\left|t-\frac12\right|}{3},
\,
\frac{2-4\left|t-\frac12\right|}{3}
\right\},
\qquad 0\le t\le1.
\]
Since $\psi_3$ is bounded, the series in \eqref{eq:T3-def-proof}
converges uniformly. Hence $T_3$ is continuous. Also
$\psi_3(0)=\psi_3(1)=0$, so
\[
T_3(0)=T_3(1)=0.
\]

The proof has three parts. First, Lemma
\ref{lem:ternary-extremizing-construction} gives the upper bound
$V_3\le T_3$. Second, Lemmas \ref{lem:local-debt} and
\ref{lem:finite-ternary-compression} prove the finite compression identity
which implies the Bellman admissibility of $T_3$, i.e. $T_3\in \mathcal{B}(1,1,3)$. Finally, Proposition
\ref{prop:T3-admissible} and Theorem \ref{mthm:1} give the lower bound
$V_3\ge T_3$.

\begin{lemma}
\label{lem:one-step-ternary}
For every $x\in[0,1]$, the following one-residual splitting has local
mean-deviation cost $\psi_3(x)$.
\begin{equation}\label{eq:T3-recursion}
T_3(x)
=
\psi_3(x)
+
\frac13T_3(\{3x\}).
\end{equation}
\end{lemma}

\begin{proof}

We will prove the result by breaking it into the following three cases. In particular we break the case $1/3\le x\le2/3$ with split
$x=\frac13(0+(3x-1)+1)$ into two corresponding costs 
$$2(1-x)/3\quad \text{for} \quad x\le1/2\quad \text {and}\quad 2x/3\quad \text{for}\quad x\ge1/2.
$$

There are three possible one-residual splittings. The middle splitting has two
affine subcases, according as $x\le1/2$ or $x\ge1/2$. This gives the following verifications.
\begin{enumerate}
\item 
If $0\le x\le1/3$, use the split
$x=\frac13(0+0+3x).$
The local cost is
\[
\frac13\bigl(|0-x|+|0-x|+|3x-x|\bigr)
=
\frac{4x}{3}.
\]
\item
If $1/3\le x\le2/3$, use the split
$x=\frac13(0+(3x-1)+1).
$
The local cost is
\[
\frac13
\bigl(
|0-x|+|3x-1-x|+|1-x|
\bigr).
\]
\item 
For $1/3\le x\le1/2$, this equals
\[
\frac13\bigl(x+(1-2x)+(1-x)\bigr)
=
\frac{2(1-x)}3.
\]
\item
For $ 
1/2\le x\le2/3$, this equals
\[
\frac13\bigl(x+(2x-1)+(1-x)\bigr)
=
\frac{2x}{3}.
\]
\item
If $2/3\le x\le1$, use the split
$
x=\frac13((3x-2)+1+1).
$
The local cost is
\[
\frac13
\bigl(
|3x-2-x|+|1-x|+|1-x|
\bigr)
=
\frac{4(1-x)}3.
\]
\end{enumerate}
Combined, these are exactly the four pieces of \eqref{eq:psi3-piecewise-proof}.
By definition, \eqref{eq:T3-recursion} follows directly
\[
T_3(x)
=
\psi_3(x)
+
\sum_{j=1}^{\infty}3^{-j}\psi_3(\{3^jx\})
=
\psi_3(x)
+
\frac13
\sum_{\ell=0}^{\infty}
3^{-\ell}\psi_3(\{3^\ell\{3x\}\})
=
\psi_3(x)+\frac13T_3(\{3x\}).
\]
\end{proof}

For $m\ge0$, define the "finite cut" version of $T_3$ by 
\[
T_{3,m}(x)
:=
\sum_{j=0}^{m-1}
3^{-j}\psi_3(\{3^jx\}),
\]
with the convention $T_{3,0}\equiv0$. Also recall
\[
N_m^{(3)}
:=
\left\{
\frac{k}{3^m}:0\le k\le3^m
\right\}.
\]
For $0\le k\le3^m$, define the scaled finite value
\begin{equation}\label{eq:scaled-Tm}
\widetilde T_m(k)
:=
3^mT_{3,m}\left(\frac{k}{3^m}\right).
\end{equation}

We shall prove the finite compression step by induction.
The only ``non-manual part'' of the proof is a finite exact verification of a
local inequality. We isolate that verification in Lemma
\ref{lem:local-debt} and prove it in Appendix
\ref{app:local-debt-verification}.\\

Define the scaled one-step profile
\[
\Lambda(u):=3\psi_3(u/3),
\qquad 0\le u\le3.
\]
Thus
\begin{equation}\label{eq:Lambda-piecewise}
\Lambda(u)
=
\begin{cases}
\dfrac{4u}{3}, & 0\le u\le1,\\[1ex]
2-\dfrac{2u}{3}, & 1\le u\le\dfrac32,\\[1ex]
\dfrac{2u}{3}, & \dfrac32\le u\le2,\\[1ex]
\dfrac{4(3-u)}{3}, & 2\le u\le3.
\end{cases}
\end{equation}

For $y=(y_1,y_2,y_3)\in\mathbb R^3$, write
\[
\mad(y)
:=
\sum_{i=1}^3
\left|
y_i-\frac{y_1+y_2+y_3}{3}
\right|.
\]

Let $a=(a_1,a_2,a_3)\in\{0,1,2\}^3$ and
$r=(r_1,r_2,r_3)\in[0,1]^3$. Put
\[
A:=a_1+a_2+a_3,
\qquad
R:=r_1+r_2+r_3,
\qquad
K:=A+R.
\]
Let
\[
q:=\min\{\lfloor K/3\rfloor,2\},
\qquad
s:=K-3q.
\]
So $0\le s\le3$. At the endpoints $K=3,6,9$, the convention is
immaterial, since $\Lambda(0)=\Lambda(3)=0$.\\

Define the local remainder
\begin{equation}\label{eq:local-remainder}
\mathcal L_a(r)
:=
\Lambda(R)
+
\sum_{i=1}^3\Lambda(a_i+r_i)
+
\mad(a+r)
-
\mad(r)
-
\Lambda(s)
-
3\Lambda(K/3).
\end{equation}
This is the exact error term that appears when one compares the compression
defect at two consecutive ternary scales.

The local remainder $\mathcal L_a$ is not always non-negative. To close
the induction we introduce a finite ``debt'' function:
\begin{equation}\label{eq:Gamma-def}
\Gamma(r)
:=
\max\Bigl\{
0,\,
-\mathcal L_b(r): b\in\{0,1,2\}^3
\Bigr\},
\qquad r\in[0,1]^3.
\end{equation}
The role of $\Gamma$ is to store the largest possible negative local
remainder at the next scale. The following lemma says that this debt is
propagated correctly through one ternary step.

\begin{lemma}
\label{lem:local-debt}
For every $a\in\{0,1,2\}^3$ and every $r\in[0,1]^3$,
\begin{equation}\label{eq:local-debt}
\Gamma(r)+\mathcal L_a(r)
\ge
3\Gamma\left(\frac{a+r}{3}\right).
\end{equation}
Moreover,
\[
\Gamma(r)=0
\qquad\text{for every }r\in\{0,1\}^3.
\]
\end{lemma}

\begin{proof}
This is the finite exact verification carried out in Appendix
\ref{app:local-debt-verification}. We explain here why the appendix proves
precisely the desired statement.

Set
\[
Y_a(r):=\Gamma(r)+\mathcal L_a(r).
\]
Since $\Gamma(r)\ge-\mathcal L_a(r)$, one has
\[
Y_a(r)\ge0.
\]
Therefore \eqref{eq:local-debt} is equivalent to the family of inequalities
\[
Y_a(r)+3\mathcal L_b\left(\frac{a+r}{3}\right)\ge0,
\qquad b\in\{0,1,2\}^3.
\]
Indeed,
\[
3\Gamma\left(\frac{a+r}{3}\right)
=
3\max\left\{
0,\,
-\mathcal L_b\left(\frac{a+r}{3}\right):
b\in\{0,1,2\}^3
\right\}.
\]
Thus it is enough to verify
\begin{equation}\label{eq:729-main}
\Gamma(r)+\mathcal L_a(r)
+
3\mathcal L_b\left(\frac{a+r}{3}\right)\ge0
\end{equation}
for every pair
$a,b\in\{0,1,2\}^3$
and every $r\in[0,1]^3$. There are $27$ choices for $a$ and
$27$ choices for $b$, hence $27\cdot27=729$ inequalities.
The appendix gives a finite computer-assisted exact verification of all
$729$ inequalities, together with the eight endpoint identities
\[
\Gamma(r)=0,\qquad r\in\{0,1\}^3.
\]
This should not be understood as numerical evidence or as a floating-point
test. After the definitions of $\Lambda$, $\mad$, $\mathcal L_a$, and
$\Gamma$ are expanded, each assertion is a first-order statement over the
ordered field of real numbers with rational coefficients. Equivalently, the
only possible obstruction would be a real point $r\in[0,1]^3$ satisfying
one of finitely many rational piecewise-affine inequalities in the wrong
direction. The Wolfram Language code in Appendix
\ref{app:local-debt-verification} uses exact quantifier elimination, via
\texttt{Resolve[..., Reals]}, to prove that no such real counterexample
exists. Thus the computation is an exact finite verification of the stated
inequalities on the whole cube $[0,1]^3$, not a sampling or numerical
approximation.
\end{proof}

We now illustrate how one of the 729 inequalities is verified below.

\begin{example}[One verification]
\label{ex:one-local-cell}
We illustrate the nature of the finite check. Consider
\[
a=b=(0,0,1).
\]
Work on the affine cell
\[
0\le r_1\le r_2\le r_3\le1,
\qquad
R:=r_1+r_2+r_3\le\frac12,
\qquad
2r_2\le r_1+r_3.
\]
On this cell all absolute values and all pieces of $\Lambda$ are fixed.
Indeed, $R\le1/2$, $1+r_3\le3/2$, and $1+R\le3/2$. Also
$2r_2\le r_1+r_3$ is exactly the condition that $r_2\le R/3$, so the
branch of $\mad(r)$ is fixed.\\

For $a=(0,0,1)$, a direct substitution into
\eqref{eq:local-remainder} gives
\[
\mathcal L_{(0,0,1)}(r)=2(r_1+r_2).
\]
Now put
\[
z:=\frac{a+r}{3}
=
\left(\frac{r_1}{3},\frac{r_2}{3},\frac{1+r_3}{3}\right).
\]
The same affine branch calculation gives
\[
\mathcal L_{(0,0,1)}(z)
=
\frac23(r_1+r_2).
\]
Thus, on this cell, the corresponding member of the family
\eqref{eq:729-main} becomes
\[
\Gamma(r)
+
\mathcal L_{(0,0,1)}(r)
+
3\mathcal L_{(0,0,1)}\left(\frac{(0,0,1)+r}{3}\right)
=
\Gamma(r)+4(r_1+r_2)\ge0.
\]
This is one typical verification. The other cells are the
same kind of rational linear calculation.
Appendix
\ref{app:local-debt-verification} performs the finite verification exactly, by quantifier elimination over the real closed field of rational coefficients.
\end{example}

\begin{lemma}
\label{lem:finite-ternary-compression}
Let $m\ge0$ and $L=3^m$. For every integer $k\in\{0,\dots,3L\}$, one has 
\begin{equation}\label{eq:finite-compression}
\widetilde T_{m+1}(k)
=
\min_{\substack{k_1,k_2,k_3\in\{0,\dots,L\}\\
k_1+k_2+k_3=k}}
\left[
\sum_{i=1}^{3}\widetilde T_m(k_i)
+
\sum_{i=1}^{3}\left|k_i-\frac{k}{3}\right|
\right].
\end{equation}
Moreover, one minimizer (up to permutation) is the one-residual triple
\[
(k_1,k_2,k_3)
=
\begin{cases}
(0,0,k), & 0\le k\le L,\\[1ex]
(0,k-L,L), & L\le k\le 2L,\\[1ex]
(k-2L,L,L), & 2L\le k\le 3L.
\end{cases}
\]
\end{lemma}

\begin{proof}
For $m\ge1$, we first record the recursion satisfied by the finite sums.
Let $H=3^{m-1}$, so that $3H=3^m$. If $a\in\{0,1,2\}$ and
$0\le r\le H$, then
\begin{equation}\label{eq:T-recursion-scaled}
\widetilde T_m(aH+r)
=
\widetilde T_{m-1}(r)
+
H\Lambda\left(a+\frac rH\right).
\end{equation}
At the endpoint $a=2,r=H$, this is interpreted by continuity; both sides
are well-defined, and the identity follows from
$\Lambda(3)=0$ and $\widetilde T_{m-1}(H)=0$. We prove by induction a strengthened statement. For $m\ge0$, $L=3^m$, and
$k_1,k_2,k_3\in\{0,\dots,L\}$, set
\[
k:=k_1+k_2+k_3
\]
and define the compression defect
\[
\Delta_m(k_1,k_2,k_3)
:=
\sum_{i=1}^3\widetilde T_m(k_i)
+
\sum_{i=1}^3\left|k_i-\frac k3\right|
-
\widetilde T_{m+1}(k).
\]
We claim that
\begin{equation}\label{eq:strengthened-defect}
\Delta_m(k_1,k_2,k_3)
\ge
L\,
\Gamma\left(
\frac{k_1}{L},
\frac{k_2}{L},
\frac{k_3}{L}
\right).
\end{equation}
Since $\Gamma\ge0$, this immediately implies the lower bound in
\eqref{eq:finite-compression}.\\

Consider the base case $m=0$, one has $L=1$ and
\[
\widetilde T_0(0)=\widetilde T_0(1)=0.
\]
The triples $(k_1,k_2,k_3)$ lie in $\{0,1\}^3$. If $k=0$ or
$k=3$, both sides vanish. If $k=1$ or $k=2$, then
\[
\sum_{i=1}^3\left|k_i-\frac k3\right|=\frac43
=
\widetilde T_1(k).
\]
Thus
\[
\Delta_0(k_1,k_2,k_3)=0.
\]
By Lemma \ref{lem:local-debt}, $\Gamma=0$ on $\{0,1\}^3$. Therefore
\eqref{eq:strengthened-defect} holds for $m=0$.\\

Assume now that \eqref{eq:strengthened-defect} has been proved at depth
$m-1$. Let $L=3^m$ and $H=3^{m-1}$, so $L=3H$. Write each $k_i$
uniquely as
\[
k_i=a_iH+r_i,
\qquad
a_i\in\{0,1,2\},
\qquad
0\le r_i\le H,
\]
using the convention
\[
a_i:=\min\{\lfloor k_i/H\rfloor,2\},
\qquad
r_i:=k_i-a_iH.
\]
Thus $k_i=L=3H$ is represented by $a_i=2,r_i=H$.
Put
\[
\rho_i:=\frac{r_i}{H},
\qquad
a:=(a_1,a_2,a_3),
\qquad
\rho:=(\rho_1,\rho_2,\rho_3).
\]
Also set
\[
A:=a_1+a_2+a_3,
\qquad
R:=\rho_1+\rho_2+\rho_3,
\qquad
K:=A+R.
\]
Let
\[
q:=\min\{\lfloor K/3\rfloor,2\},
\qquad
s:=K-3q.
\]
Then
$k=HK,$
and 
$k=3qH+sH.$
Using \eqref{eq:T-recursion-scaled}, we compute
\[
\sum_{i=1}^3\widetilde T_m(k_i)
=
\sum_{i=1}^3\widetilde T_{m-1}(r_i)
+
H\sum_{i=1}^3\Lambda(a_i+\rho_i).
\]
Moreover,
\[
\sum_{i=1}^3\left|k_i-\frac k3\right|
=
H\mad(a+\rho).
\]
For the target term,
\[
\widetilde T_{m+1}(k)
=
\widetilde T_m(sH)
+
3H\Lambda(K/3).
\]
On the other hand,
\[
\widetilde T_m(r_1+r_2+r_3)
=
\widetilde T_m(HR).
\]
Since
\[
s-R=A-3q
\]
is an integer, $s$ and $R$ have the same fractional part. 
Both $HR=r_1+r_2+r_3$ and $Hs=k-3qH$ are integers in $[0,3H]$.
Applying \eqref{eq:T-recursion-scaled} to $HR$ and $Hs$ gives
\[
\widetilde T_m(HR)-\widetilde T_m(Hs)
=
H\bigl(\Lambda(R)-\Lambda(s)\bigr).
\]
Combining the preceding identities yields the exact defect recursion
\begin{equation}\label{eq:defect-recursion}
\Delta_m(k_1,k_2,k_3)
=
\Delta_{m-1}(r_1,r_2,r_3)
+
H\mathcal L_a(\rho).
\end{equation}

By the induction hypothesis,
\[
\Delta_{m-1}(r_1,r_2,r_3)
\ge
H\Gamma(\rho).
\]
Therefore \eqref{eq:defect-recursion} gives
\[
\Delta_m(k_1,k_2,k_3)
\ge
H\bigl(\Gamma(\rho)+\mathcal L_a(\rho)\bigr).
\]
By the local debt inequality, Lemma \ref{lem:local-debt},
\[
\Gamma(\rho)+\mathcal L_a(\rho)
\ge
3\Gamma\left(\frac{a+\rho}{3}\right).
\]
Since
\[
\frac{a_i+\rho_i}{3}
=
\frac{a_iH+r_i}{3H}
=
\frac{k_i}{L},
\]
we obtain
\[
\Delta_m(k_1,k_2,k_3)
\ge
3H
\Gamma\left(
\frac{k_1}{L},
\frac{k_2}{L},
\frac{k_3}{L}
\right)
=
L
\Gamma\left(
\frac{k_1}{L},
\frac{k_2}{L},
\frac{k_3}{L}
\right).
\]
This proves the strengthened induction and hence the lower bound in
\eqref{eq:finite-compression}.\\

It remains to show that the displayed one-residual triples attain equality. We split into five cases as in Lemma \ref{lem:one-step-ternary}:
\begin{enumerate}
\item 
If $0\le k\le L$, then the triple $(0,0,k)$ gives
\[
\sum_{i=1}^{3}\widetilde T_m(k_i)
+
\sum_{i=1}^{3}\left|k_i-\frac{k}{3}\right|
=
\widetilde T_m(k)
+
\frac{4k}{3}.
\]
Since $k/(3L)\le1/3$, this equals $\widetilde T_{m+1}(k)$ by the first
piece of $\psi_3$.
\item 
If $L\le k\le2L$, then the triple $(0,k-L,L)$ gives
\[
\widetilde T_m(k-L)
+
\left|0-\frac{k}{3}\right|
+
\left|k-L-\frac{k}{3}\right|
+
\left|L-\frac{k}{3}\right|.
\]
\item 
For $L\le k\le3L/2$, this equals
\[
\widetilde T_m(k-L)+2L-\frac{2k}{3}.
\]
This is the first middle piece of $\widetilde T_{m+1}(k)$.
\item 
For $3L/2\le k\le2L$, this equals
\[
\widetilde T_m(k-L)+\frac{2k}{3}.
\]
This is the second middle piece of $\widetilde T_{m+1}(k)$.
\item 
Finally, if $2L\le k\le3L$, then the triple $(k-2L,L,L)$ gives
\[
\widetilde T_m(k-2L)+\frac{4(3L-k)}{3},
\]
which is exactly $\widetilde T_{m+1}(k)$ by the last piece of
$\psi_3$.
\end{enumerate}

Thus the lower bound is sharp, and \eqref{eq:finite-compression} follows.
\end{proof}

We now show that $T_3$ is an admissible Bellman function.

\begin{proposition}
\label{prop:T3-admissible}
The function $T_3$ belongs to $\mathcal B(1,1,3)$. Equivalently, for
all $x_1,x_2,x_3\in[0,1]$, one has
\begin{equation}\label{eq:T3-admissibility}
T_3(\bar x)
\le
\frac13\sum_{i=1}^{3}T_3(x_i)
+
\frac13\sum_{i=1}^{3}|x_i-\bar x|,\quad 
\bar x=\frac{x_1+x_2+x_3}{3}.
\end{equation}
\end{proposition}

\begin{proof}
First assume $x_1,x_2,x_3\in N_m^{(3)}$. Write
\[
x_i=\frac{k_i}{3^m},
\qquad
0\le k_i\le3^m,
\quad 
k:=k_1+k_2+k_3
,\quad 
\bar x=\frac{k}{3^{m+1}}.
\]
By Lemma \ref{lem:finite-ternary-compression},
\[
\widetilde T_{m+1}(k)
\le
\sum_{i=1}^{3}\widetilde T_m(k_i)
+
\sum_{i=1}^{3}\left|k_i-\frac{k}{3}\right|.
\]
Dividing by $3^{m+1}$, and using
\[
\widetilde T_{m+1}(k)
=
3^{m+1}T_{3,m+1}(\bar x),
\qquad
\widetilde T_m(k_i)
=
3^mT_{3,m}(x_i),
\]
gives
\[
T_{3,m+1}(\bar x)
\le
\frac13\sum_{i=1}^{3}T_{3,m}(x_i)
+
\frac13\sum_{i=1}^{3}|x_i-\bar x|.
\]
Since $x_i\in N_m^{(3)}$, one has
\[
T_3(x_i)=T_{3,m}(x_i).
\]
Since $\bar x\in N_{m+1}^{(3)}$, one has
\[
T_3(\bar x)=T_{3,m+1}(\bar x).
\]
Therefore \eqref{eq:T3-admissibility} holds for triples of ternary rational
points.\\

Finally, $N^{(3)}=\bigcup_m N_m^{(3)}$ is dense in $[0,1]$, and
$T_3$ is continuous. Passing to limits gives
\eqref{eq:T3-admissibility} for all $x_1,x_2,x_3\in[0,1]$. Since
$T_3(0)=T_3(1)=0$, we conclude that $T_3\in\mathcal B(1,1,3).$
\end{proof}

\begin{lemma}[Extremizing ternary construction]
\label{lem:ternary-extremizing-construction}
For every $x\in[0,1]$, there exists a measurable set
$E_x\subset[0,1)$ such that
$
|E_x|=x
$
and
\[
\|S_1(\mathbbm 1_{E_x})\|_1
=
T_3(x).
\]
Consequently, $V_3(x)\le T_3(x)$.
\end{lemma}

\begin{proof}
The cases $x=0$ and $x=1$ are immediate: take
$E_0=\emptyset$ and $E_1=[0,1)$. Assume now that $0<x<1$.\\

We construct $E_x$ recursively. At each step there is at most one
residual ternary interval. If the current relative density is $r$, we
apply the one-residual splitting from Lemma \ref{lem:one-step-ternary}.
\begin{enumerate}
\item 
If $0<r<1/3$, declare two children empty and continue inside the third
child with relative density $3r$.

\item
If $1/3<r<2/3$, declare one child empty, one child full, and continue
inside the remaining child with relative density $3r-1$.
\item If $2/3<r<1$, declare two children full and continue inside the remaining
child with relative density $3r-2$.

\end{enumerate}
At the boundary values $r=1/3$ and $r=2/3$, the same splitting gives
only terminal children after the split; equivalently, the next residual
density is $\{3r\}=0$, and the construction stops. Thus in all cases the
next residual density is
\[
\{3r\}.
\]

Let $I_j$ be the residual interval after $j$ steps, assuming we do not stop. Then
\[
|I_j|=3^{-j},
\qquad
r_j=\{3^j x\}.
\]
At the next splitting, the density on $I_j$ is $r_j$, and the local
cost is $\psi_3(r_j)$. Therefore
\[
\int_{I_j}|d_{j+1}|\,dy
=
|I_j|\psi_3(r_j)
=
3^{-j}\psi_3(\{3^j x\}).
\]
Outside $I_j$, all intervals have already been declared full or empty,
so all later martingale differences vanish there.\\

The residual intervals are nested and have lengths going to zero. Hence
the union of all intervals declared full defines a measurable set
$E_x\subset[0,1)$. After $N$ steps, the construction has already
declared a finite collection of ternary intervals full, and one residual
interval of length $3^{-N}$ remains, with relative density
$\{3^N x\}$. Thus
\[
x
=
\text{mass declared full after }N\text{ steps}
+
3^{-N}\{3^N x\}.
\]
Letting $N\to\infty$, and using $0\le \{3^N x\}\le1$, gives
\[
|E_x|=x.
\]

Finally, since all terms are non-negative,
\[
S_1(\mathbbm 1_{E_x})
=
\sum_{m\ge1}|d_m|
\]
pointwise as an increasing limit of partial sums. By monotone convergence,
\[
\|S_1(\mathbbm 1_{E_x})\|_1
=
\sum_{j=0}^{\infty}
\int_{I_j}|d_{j+1}|\,dy.
\]
Using the cost computation above,
\[
\|S_1(\mathbbm 1_{E_x})\|_1
=
\sum_{j=0}^{\infty}
3^{-j}\psi_3(\{3^jx\})
=
T_3(x).
\]
This proves the lemma.
\end{proof}
Before concluding the proof we shall present the asymptotics behavior.
\begin{lemma}
\label{lem:T3-asymptotics}
For $0<x<1$,
\[
T_3(x)
\asymp
x^*\log\frac1{x^*},
\qquad
x^*=\min(x,1-x).
\]
\end{lemma}

\begin{proof}
For $0\le t\le1$, the profile $\psi_3$ is comparable to the distance
from $t$ to the nearest integer:
\begin{equation}\label{eq:psi3-distance-comparison}
\frac23\,\dist(t,\mathbb Z)
\le
\psi_3(t)
\le
\frac43\,\dist(t,\mathbb Z).
\end{equation}
Therefore
\[
T_3(x)
\asymp
\sum_{j=0}^{\infty}
3^{-j}\dist(3^jx,\mathbb Z).
\]
Denote the last sum by
\[
W_3(x)
:=
\sum_{j=0}^{\infty}
3^{-j}\dist(3^jx,\mathbb Z).
\]
Since $\psi_3(1-t)=\psi_3(t)$, we have
$
T_3(1-x)=T_3(x)$.
It is therefore enough to consider $0<x\le1/2$.

Assume first that $0<x\le1/3$. Choose $M\ge1$ such that
\[
3^{-(M+1)}<x\le3^{-M}.
\]
For $0\le j\le M-1$, one has $3^jx\le1/3$, and hence
$\dist(3^jx,\mathbb Z)=3^jx.$
Thus
\[
W_3(x)
\ge
\sum_{j=0}^{M-1}3^{-j}3^jx
=
Mx
\gtrsim
x\log\frac1x.
\]

For the upper bound, use
$\dist(3^jx,\mathbb Z)
\le
\min\{3^jx,1/2\}.$
Then
\[
W_3(x)
\le
\sum_{j=0}^{M}3^{-j}3^jx
+
\frac12\sum_{j=M+1}^{\infty}3^{-j}
\le
(M+1)x+C3^{-M}.
\]
Since $3^{-M}\lesssim x$, this gives
$W_3(x)
\lesssim
x\log\frac1x.$
Therefore
\[
T_3(x)
\asymp
x\log\frac1x,
\qquad
0<x\le\frac13.
\]

On the compact interval $1/3\le x\le1/2$, the function
$x\log(1/x)$ is bounded above and below by positive absolute constants.
Also $T_3(x)\ge\psi_3(x)$, and $\psi_3$ is bounded below by a positive
constant on $[1/3,1/2]$. Hence the same comparison holds on
$0<x\le1/2$. By symmetry,
\[
T_3(x)
\asymp
x^*\log\frac1{x^*},
\qquad
0<x<1.
\]
\end{proof}

\begin{proof}[Proof of Theorem \ref{mthm:2}]
By Proposition \ref{prop:T3-admissible},
$T_3\in\mathcal B(1,1,3).$
Therefore Theorem \ref{mthm:1}, with $n=3$ and $\alpha=\beta=1$,
gives
\[
\|S_1(\mathbbm 1_A)\|_1
\ge
T_3(|A|)
\]
for every measurable set $A\subset[0,1)$. Hence, if $|A|=x$,
\[
\|S_1(\mathbbm 1_A)\|_1
\ge
T_3(x).
\]
Taking the infimum over all such $A$, we obtain
\[
V_3(x) =\inf_{A}\|S_1(\mathbbm 1_A)\|_1 \ge T_3(x).
\]

Conversely, Lemma \ref{lem:ternary-extremizing-construction} gives a
measurable set $E_x\subset[0,1)$ such that
$|E_x|=x$
and
\[
\|S_1(\mathbbm 1_{E_x})\|_1=T_3(x).
\]
Therefore
$
V_3(x)\le T_3(x).
$
Combining the two inequalities gives
\[
V_3(x)=T_3(x).
\]

Now define the pointwise Bellman envelope by
\[
B_{1,1,3}(x)
:=
\sup\{F(x):F\in\mathcal B(1,1,3)\}.
\]
Since $T_3\in\mathcal B(1,1,3)$, we have
\[
B_{1,1,3}(x)\ge T_3(x).
\]
Conversely, let $F\in\mathcal B(1,1,3)$. Applying Theorem
\ref{mthm:1} to the extremizing set $E_x$ from Lemma
\ref{lem:ternary-extremizing-construction}, we get
\[
F(x)
=
F(|E_x|)
\le
\|S_1(\mathbbm 1_{E_x})\|_1
=
T_3(x).
\]
Taking the supremum over all admissible $F$, we obtain
\[
B_{1,1,3}(x)\le T_3(x).
\]
Therefore $B_{1,1,3}(x)=T_3(x).
$
Thus
\[
B_{1,1,3}(x)
=
V_3(x)
=
\sum_{j=0}^{\infty}3^{-j}\psi_3(\{3^jx\}).
\]
Finally, Lemma \ref{lem:T3-asymptotics} gives
\[
T_3(x)
\asymp
x^*\log\frac1{x^*},
\qquad
x^*=\min(x,1-x),
\qquad
0<x<1.
\]
This completes the proof.
\end{proof}

\section{Proof of Theorem \ref{Thm:digit-sum-Pn}}
\label{sec:proof-gtilde-n}

Fix an integer $n\ge2$. Recall that for $k\ge0$, we write
\[
k=\sum_{\nu\ge0}\varepsilon_\nu(k)n^\nu,
\qquad
\varepsilon_\nu(k)\in\{0,1,\dots,n-1\},
\]
and define
\[
\sigma_n(k):=\sum_{\nu\ge0}\varepsilon_\nu(k),
\qquad
G_n(k):=\sum_{j=0}^{k-1}\sigma_n(j),
\qquad
G_n(0):=0.
\]
For $x\in N_m$, recall that
\begin{equation}\label{eq:Bm-proof-def}
B_m^{(n)}(x)
:=
mx-\frac{2}{(n-1)n^m}G_n(n^m x).
\end{equation}

\begin{lemma}[Digit-sum recursions]
\label{lem:digit-sum-recursions}
For all integers $q\ge0$ and all $r\in\{0,1,\dots,n-1\}$, one has
\[
\sigma_n(nq+r)=\sigma_n(q)+r,
\]
and
\begin{equation}\label{eq:G-recursion}
G_n(nq+r)
=
nG_n(q)
+
\frac{n(n-1)}2q
+
r\sigma_n(q)
+
\frac{r(r-1)}2.
\end{equation}
In particular,
\begin{equation}\label{eq:G-nq}
G_n(nq)
=
nG_n(q)
+
\frac{n(n-1)}2q.
\end{equation}
Moreover, for every $m\ge0$,
\begin{equation}\label{eq:G-nm}
G_n(n^m)
=
m\,\frac{n^m(n-1)}2.
\end{equation}
\end{lemma}

\begin{proof}
Multiplication by $n$ shifts the base-$n$ expansion of $q$ one place
to the left. Adding $r$ inserts $r$ as the least significant digit.
Therefore
\[
\sigma_n(nq+r)=\sigma_n(q)+r.
\]

For $G_n$, split the defining sum into complete blocks of length $n$
and one remaining tail:
\[
G_n(nq+r)
=
\sum_{a=0}^{q-1}\sum_{\rho=0}^{n-1}\sigma_n(na+\rho)
+
\sum_{\rho=0}^{r-1}\sigma_n(nq+\rho).
\]
Using $\sigma_n(na+\rho)=\sigma_n(a)+\rho$, the complete blocks give
\[
\sum_{a=0}^{q-1}
\left(
n\sigma_n(a)+\frac{n(n-1)}2
\right)
=
nG_n(q)
+
\frac{n(n-1)}2q.
\]
The tail gives
\[
\sum_{\rho=0}^{r-1}(\sigma_n(q)+\rho)
=
r\sigma_n(q)+\frac{r(r-1)}2.
\]
This proves \eqref{eq:G-recursion}. Taking $r=0$ gives
\eqref{eq:G-nq}. Finally, iterating \eqref{eq:G-nq} and using $G_n(1)=0$
gives \eqref{eq:G-nm}.
\end{proof}

\begin{lemma}[Compatibility of finite levels]
\label{lem:Bm-compatibility}
For every $m\ge0$,
\[
B_{m+1}^{(n)}\big|_{N_m}=B_m^{(n)}.
\]
Consequently, if $x\in N_{m_0}$, then
\[
B_m^{(n)}(x)=B_{m_0}^{(n)}(x)
\qquad
\text{for every }m\ge m_0.
\]
In particular, the limit
\[
P^{(n)}(x):=\lim_{m\to\infty}B_m^{(n)}(x),
\qquad x\in N,
\]
exists and is finite.
\end{lemma}

\begin{proof}
Let $x=k/n^m\in N_m$. Then $x=nk/n^{m+1}$, and by
\eqref{eq:G-nq},
\[
G_n(nk)
=
nG_n(k)
+
\frac{n(n-1)}2k.
\]
Therefore
\begin{align*}
B_{m+1}^{(n)}(x)
&=
(m+1)\frac{k}{n^m}
-
\frac{2}{(n-1)n^{m+1}}G_n(nk)
\\
&=
(m+1)\frac{k}{n^m}
-
\frac{2}{(n-1)n^{m+1}}
\left(
nG_n(k)+\frac{n(n-1)}2k
\right)
\\
&=
m\frac{k}{n^m}
-
\frac{2}{(n-1)n^m}G_n(k)
=
B_m^{(n)}(x).
\end{align*}
The stabilization and the existence of the limit follow by iteration.
\end{proof}

\begin{lemma}[Endpoint values]
\label{lem:Pn-endpoints}
The function $P^{(n)}$ satisfies
\[
P^{(n)}(0)=P^{(n)}(1)=0.
\]
\end{lemma}

\begin{proof}
Clearly $B_m^{(n)}(0)=0$. Also, by \eqref{eq:G-nm},
\[
G_n(n^m)
=
m\,\frac{n^m(n-1)}2.
\]
Hence
\[
B_m^{(n)}(1)
=
m-\frac{2}{(n-1)n^m}G_n(n^m)
=
m-\frac{2}{(n-1)n^m}
\cdot
m\frac{n^m(n-1)}2
=
0.
\]
Passing to the stabilized limit gives the claim.
\end{proof}

We recall the following summatory digit-sum inequality (introduced in its original notation first and then adapted to ours) 

\begin{lemma}[Allouche--Stipulanti {\cite[Theorem~4.2]{AlloucheStipulanti2025Summing}}]
\label{lem:AS-summatory-digit-sum}
Let $b\ge2$, and let $s_b(m)$ denote the sum of the base-$b$ digits of
$m$. Put
\[
S_b(M):=\sum_{j=0}^{M-1}s_b(j),\qquad S_b(0)=0.
\]
If $1\le r\le b$ and $0\le m_1\le\cdots\le m_r$, then
\[
\sum_{i=1}^r S_b(m_i)+\sum_{i=1}^{r-1}(r-i)m_i
\le
S_b\left(\sum_{i=1}^r m_i\right).
\]
\end{lemma}

We shall apply this only with $b=r=n$. Since in our notation
$S_n=G_n$, we obtain
\[
\sum_{i=1}^n G_n(k_i)+\sum_{i=1}^{n-1}(n-i)k_i
\le
G_n\left(\sum_{i=1}^n k_i\right).
\]

\begin{proposition}[Summatory digit-sum inequality]
\label{prop:digit-sum-ineq}
Let
$ 0\le k_1\le k_2\le\cdots\le k_n $
be integers, and set
$
K:=\sum_{i=1}^n k_i.
$
Then
\begin{equation}\label{eq:digit-sum-ineq}
G_n(K)-\sum_{i=1}^nG_n(k_i)
\ge
\sum_{i=1}^{n-1}(n-i)k_i.
\end{equation}
\end{proposition}

\begin{proof}

This is the summatory base-$n$ digit-sum inequality of
Allouche--Stipulanti in Lemma \ref{lem:AS-summatory-digit-sum}, applied with base $b=n$ and with $r=n$. In the present
notation, their inequality says
\[
\sum_{i=1}^nG_n(k_i)
+
\sum_{i=1}^{n-1}(n-i)k_i
\le
G_n\left(\sum_{i=1}^n k_i\right),
\]
which is exactly \eqref{eq:digit-sum-ineq}.
\end{proof}

\begin{lemma}[Mean-deviation $n$-point inequality on $N$]
\label{lem:Pn-npoint-on-N}
For every $x_1,\dots,x_n\in N$,
one has
\begin{equation}\label{eq:Pn-npoint-on-N}
P^{(n)}(\bar x)
\le
\frac1n\sum_{i=1}^nP^{(n)}(x_i)
+
\frac1n\sum_{i=1}^n|x_i-\bar x|,\qquad
\bar x=\frac1n\sum_{i=1}^n x_i.
\end{equation}
\end{lemma}

\begin{proof}
Both sides of \eqref{eq:Pn-npoint-on-N} are symmetric in the points
$x_1,\dots,x_n$. Hence we may assume
\[
0\le x_1\le x_2\le\cdots\le x_n\le1.
\]
Choose $m$ large enough so that $x_i\in N_m$ for every $i$. Let $0\le k_1\le k_2\le\cdots\le k_n\le n^m.$ Write
\[
x_i=\frac{k_i}{n^m},
\qquad
K:=\sum_{i=1}^n k_i,\quad \text{and}\quad 
\bar x=\frac{K}{n^{m+1}}.
\]
By Lemma \ref{lem:Bm-compatibility},
\[
P^{(n)}(\bar x)
=
B_{m+1}^{(n)}\left(\frac{K}{n^{m+1}}\right),
\qquad
P^{(n)}(x_i)
=
B_m^{(n)}\left(\frac{k_i}{n^m}\right).
\]
Using the definition \eqref{eq:Bm-proof-def}, we get
\begin{align*}
P^{(n)}(\bar x)
-
\frac1n\sum_{i=1}^nP^{(n)}(x_i)
&=
\frac{1}{n^{m+1}}
\left[
K
-
\frac{2}{n-1}
\left(
G_n(K)-\sum_{i=1}^nG_n(k_i)
\right)
\right].
\end{align*}
By Proposition \ref{prop:digit-sum-ineq},
\[
G_n(K)-\sum_{i=1}^nG_n(k_i)
\ge
\sum_{i=1}^{n-1}(n-i)k_i.
\]
Therefore
\begin{align*}
P^{(n)}(\bar x)
-
\frac1n\sum_{i=1}^nP^{(n)}(x_i)
&\le
\frac{1}{n^{m+1}}
\left[
K
-
\frac{2}{n-1}\sum_{i=1}^{n-1}(n-i)k_i
\right].
\end{align*}
The expression in brackets can be written as
\[
\sum_{i=1}^n a_i k_i,
\qquad
a_i:=\frac{2i-n-1}{n-1}.
\]
Indeed, for $i<n$,
\[
1-\frac{2(n-i)}{n-1}
=
\frac{2i-n-1}{n-1},
\]
and for $i=n$, the same formula gives $a_n=1$.
Now observe that
\[
\sum_{i=1}^n a_i=0,
\qquad
|a_i|\le1
\quad
\text{for every }i.
\]
Set
$
\mu:=\frac Kn.
$ Since $\sum_i a_i=0$, we have
\[
\sum_{i=1}^n a_i k_i
=
\sum_{i=1}^n a_i(k_i-\mu).
\]
Hence
\[
\sum_{i=1}^n a_i k_i
\le
\sum_{i=1}^n |a_i|\,|k_i-\mu|
\le
\sum_{i=1}^n |k_i-\mu|.
\]
Therefore
\[
P^{(n)}(\bar x)
-
\frac1n\sum_{i=1}^nP^{(n)}(x_i)
\le
\frac{1}{n^{m+1}}
\sum_{i=1}^n\left|k_i-\frac Kn\right|.
\]
Finally,
\[
\frac{1}{n^{m+1}}
\sum_{i=1}^n\left|k_i-\frac Kn\right|
=
\frac1n
\sum_{i=1}^n
\left|
\frac{k_i}{n^m}
-
\frac{K}{n^{m+1}}
\right|
=
\frac1n
\sum_{i=1}^n|x_i-\bar x|.
\]
This proves \eqref{eq:Pn-npoint-on-N}.
\end{proof}

In the next lemma we write the Takagi-type representation of $P^{(n)}$. We use $\eta$ for the generator  instead of $\psi$ to avoid confusion with the sharp ternary profile from Theorem \ref{mthm:2}).

\begin{lemma}[Takagi-type representation]
\label{lem:Pn-Takagi-representation}
Let $\eta_n:[0,1]\to[0,\infty)$ be the continuous piecewise linear
function determined by
\[
\eta_n\left(\frac{\rho}{n}\right)
=
\frac{\rho(n-\rho)}{n(n-1)},
\qquad
\rho=0,1,\dots,n.
\]
Then, for every $x\in N$,
\begin{equation}\label{eq:seriesPn}
    P^{(n)}(x)
=
\sum_{j=0}^{\infty}n^{-j}\eta_n(\{n^jx\}).
\end{equation}
Moreover, the series converges uniformly on $[0,1]$. Hence it defines a
continuous extension of $P^{(n)}$ to $[0,1]$, still denoted by
$P^{(n)}$.
\end{lemma}

\begin{proof}
For $m\ge0$, define the finite partial sums
\[
T_m^{(n)}(x)
:=
\sum_{j=0}^{m-1}n^{-j}\eta_n(\{n^jx\}),
\]
with the convention $T_0^{(n)}\equiv0$. We first prove that
\[
T_m^{(n)}(x)=B_m^{(n)}(x),
\qquad x\in N_m.
\]

The claim is clear for $m=0$. Assume it is known at level $m-1$.
Let $0\le k<n^m$, and write
\[
k=tn+\rho,
\qquad
0\le \rho\le n-1.
\]
Then
\[
\frac{k}{n^m}
=
\left(1-\frac{\rho}{n}\right)\frac{t}{n^{m-1}}
+
\frac{\rho}{n}\frac{t+1}{n^{m-1}}.
\]
Since $T_{m-1}^{(n)}$ is linear on every $n$-adic interval of level
$m-1$,
\[
T_{m-1}^{(n)}\left(\frac{k}{n^m}\right)
=
\left(1-\frac{\rho}{n}\right)
T_{m-1}^{(n)}\left(\frac{t}{n^{m-1}}\right)
+
\frac{\rho}{n}
T_{m-1}^{(n)}\left(\frac{t+1}{n^{m-1}}\right).
\]
Also,
\[
\left\{n^{m-1}\frac{k}{n^m}\right\}
=
\frac{\rho}{n}.
\]
Therefore
\begin{align*}
T_m^{(n)}\left(\frac{k}{n^m}\right)
&=
\left(1-\frac{\rho}{n}\right)
T_{m-1}^{(n)}\left(\frac{t}{n^{m-1}}\right)
+
\frac{\rho}{n}
T_{m-1}^{(n)}\left(\frac{t+1}{n^{m-1}}\right)
+
n^{-(m-1)}\eta_n\left(\frac{\rho}{n}\right)
\\
&=
\left(1-\frac{\rho}{n}\right)
B_{m-1}^{(n)}\left(\frac{t}{n^{m-1}}\right)
+
\frac{\rho}{n}
B_{m-1}^{(n)}\left(\frac{t+1}{n^{m-1}}\right)
+
\frac{\rho(n-\rho)}{(n-1)n^m}.
\end{align*}
On the other hand, using \eqref{eq:G-recursion} and
$G_n(t+1)=G_n(t)+\sigma_n(t)$, a direct substitution in
\eqref{eq:Bm-proof-def} gives
\begin{align*}
B_m^{(n)}\left(\frac{tn+\rho}{n^m}\right)
&=
\left(1-\frac{\rho}{n}\right)
B_{m-1}^{(n)}\left(\frac{t}{n^{m-1}}\right)
+
\frac{\rho}{n}
B_{m-1}^{(n)}\left(\frac{t+1}{n^{m-1}}\right)
+
\frac{\rho(n-\rho)}{(n-1)n^m}.
\end{align*}
Thus
\[
T_m^{(n)}\left(\frac{k}{n^m}\right)
=
B_m^{(n)}\left(\frac{k}{n^m}\right)
\]
for $0\le k<n^m$. The endpoint $k=n^m$ is immediate, since both sides
vanish at $1$. This completes the induction.\\

Now let $x\in N$, and choose $m_0$ such that $x\in N_{m_0}$. Then
\[
P^{(n)}(x)=B_{m_0}^{(n)}(x)=T_{m_0}^{(n)}(x).
\]
For $j\ge m_0$, $n^jx$ is an integer, and hence
$
\eta_n(\{n^jx\})=\eta_n(0)=0$.
Therefore
\[
T_{m_0}^{(n)}(x)
=
\sum_{j=0}^{\infty}n^{-j}\eta_n(\{n^jx\}).
\]
This proves the desired representation on $N$.\\

Finally, $\eta_n$ is bounded, so
\[
\sum_{j=0}^{\infty}n^{-j}\eta_n(\{n^jx\})
\]
converges uniformly on $[0,1]$. Hence the series defines a continuous
extension of $P^{(n)}$ from $N$ to $[0,1]$.
\end{proof}

\begin{lemma}[Extension of the $n$-point inequality]
\label{lem:Pn-npoint-extension}
For all $x_1,\dots,x_n\in[0,1]$, the continuous extension of $P^{(n)}$ satisfies
\[
P^{(n)}(\bar x)
\le
\frac1n\sum_{i=1}^nP^{(n)}(x_i)
+
\frac1n\sum_{i=1}^n|x_i-\bar x|,
\qquad 
\bar x=\frac1n\sum_{i=1}^n x_i.
\]
\end{lemma}

\begin{proof}
Choose sequences $x_i^{(\ell)}\in N$ such that
$x_i^{(\ell)}\to x_i$ as 
$\ell\to\infty.$
Set
\[
\bar x^{(\ell)}
:=
\frac1n\sum_{i=1}^n x_i^{(\ell)}.
\]
Then $\bar x^{(\ell)}\in N$ and $\bar x^{(\ell)}\to\bar x.$
Applying Lemma \ref{lem:Pn-npoint-on-N} to
$x_1^{(\ell)},\dots,x_n^{(\ell)}$, we get
\[
P^{(n)}(\bar x^{(\ell)})
\le
\frac1n\sum_{i=1}^nP^{(n)}(x_i^{(\ell)})
+
\frac1n\sum_{i=1}^n|x_i^{(\ell)}-\bar x^{(\ell)}|.
\]
Letting $\ell\to\infty$, and using the continuity of $P^{(n)}$, gives
the desired inequality.
\end{proof}

\begin{lemma}[Symmetry of $P^{(n)}$]
\label{lem:Pn-symmetry}
For every $n\ge2$, the continuous extension of $P^{(n)}$ satisfies
\[
P^{(n)}(1-x)=P^{(n)}(x),
\qquad 0\le x\le1.
\]
\end{lemma}

\begin{proof}
By Lemma \ref{lem:Pn-Takagi-representation}, $P^{(n)}$ has the power-series representation ... \eqref{eq:seriesPn}, with the continuous piecewise linear generator  $\eta_n$ determined by
\[
\eta_n\left(\frac{\rho}{n}\right)
=
\frac{\rho(n-\rho)}{n(n-1)},
\qquad
\rho=0,1,\dots,n.
\]
The nodal values are symmetric, since
\[
\eta_n\left(1-\frac{\rho}{n}\right)
=
\eta_n\left(\frac{n-\rho}{n}\right)
=
\frac{(n-\rho)\rho}{n(n-1)}
=
\eta_n\left(\frac{\rho}{n}\right).
\]
Because $\eta_n$ is linear on each interval
$[\rho/n,(\rho+1)/n]$, it follows that
\[
\eta_n(1-t)=\eta_n(t),
\qquad 0\le t\le1.
\]
Now fix $x\in[0,1]$. For each $j\ge0$, if $n^jx\in\mathbb Z$, then
\[
\{n^j(1-x)\}=0=\{n^jx\},\quad \text{otherwise}\quad
\{n^j(1-x)\}=1-\{n^jx\}.
\]
Using the symmetry of $\eta_n$, and also $\eta_n(0)=\eta_n(1)=0$, we get
\[
\eta_n(\{n^j(1-x)\})
=
\eta_n(\{n^jx\})
\]
for every $j\ge0$. Therefore, term by term,
\[
P^{(n)}(1-x)
=
\sum_{j=0}^{\infty} n^{-j}\eta_n(\{n^j(1-x)\})
=
\sum_{j=0}^{\infty} n^{-j}\eta_n(\{n^jx\})
=
P^{(n)}(x).
\]
\end{proof}
\begin{lemma}[Endpoint growth]
\label{lem:Pn-growth}
For $0<x<1$,
\[
P^{(n)}(x)
\asymp_n
x^*\log\left(\frac1{x^*}\right),\qquad
x^*:=\min\{x,1-x\}.
\]
\end{lemma}

\begin{proof}
Let
\[
\phi(t):=\dist(t,\mathbb Z),
\qquad 0\le t\le1.
\]
The generator $\eta_n$ is continuous, non-negative, positive on
$(0,1)$, and vanishes linearly at $0$ and $1$. Hence there exist
constants $0<c_n\le C_n<\infty$, depending only on $n$, such that
\begin{equation}\label{eq:psi-dist-comparison}
c_n\phi(t)
\le
\eta_n(t)
\le
C_n\phi(t),
\qquad
0\le t\le1.
\end{equation}
By Lemma \ref{lem:Pn-Takagi-representation},
\[
P^{(n)}(x)
\asymp_n
W_n(x),
\qquad
W_n(x):=
\sum_{j=0}^{\infty}
n^{-j}\dist(n^jx,\mathbb Z).
\]
Since $\eta_n(1-t)=\eta_n(t)$, we have
\[
P^{(n)}(1-x)=P^{(n)}(x).
\]
Thus it is enough to prove the estimate for $0<x\le1/2$.

For the upper bound, set
$
L:=\left\lfloor \log_n\left(\frac1x\right)\right\rfloor.
$
Using
$
\dist(n^jx,\mathbb Z)\le \min\{n^jx,1/2\},
$
we get
\[
n^{-j}\dist(n^jx,\mathbb Z)
\le
\min\{x,\tfrac12 n^{-j}\}.
\]
Therefore
\[
W_n(x)
\le
\sum_{j=0}^{L}x
+
\frac12\sum_{j=L+1}^{\infty}n^{-j}
\le
(L+1)x+C_n n^{-L}.
\]
Since $n^{-L}\lesssim_n x$, this gives
\[
W_n(x)
\lesssim_n
x\log\left(\frac1x\right).
\]

For the lower bound, set
$
M:=\left\lfloor \log_n\left(\frac1{2x}\right)\right\rfloor.
$
For $0\le j\le M$, one has $n^jx\le1/2$. Hence
$\dist(n^jx,\mathbb Z)=n^jx,$
and so
$
n^{-j}\dist(n^jx,\mathbb Z)=x.
$
Therefore
\[
W_n(x)
\ge
\sum_{j=0}^{M}x
=
(M+1)x
\gtrsim_n
x\log\left(\frac1x\right).
\]
Combining the upper and lower bounds gives, and using symmetry over $x\in (0,1/2]$, we get
\[
P^{(n)}(x)
\asymp_n
x^*\log\left(\frac1{x^*}\right),
\qquad
0<x<1.
\]
\end{proof}

\begin{proof}[Proof of Theorem \ref{Thm:digit-sum-Pn}]
Item $(1)$ is Lemma \ref{lem:Bm-compatibility}. Item $(2)$ follows from
Lemma \ref{lem:Pn-endpoints} and the non-negativity of the Takagi-type
series in Lemma \ref{lem:Pn-Takagi-representation}. Item $(3)$ is Lemma
\ref{lem:Pn-npoint-on-N}.
Item $(5)$ is Lemma \ref{lem:Pn-Takagi-representation}, item $(6)$ is
Lemma \ref{lem:Pn-symmetry}, and item $(7)$ is Lemma \ref{lem:Pn-growth}.

It remains only to justify item $(4)$. By Lemma
\ref{lem:Pn-Takagi-representation}, $P^{(n)}$ extends continuously to
$[0,1]$. By Lemma \ref{lem:Pn-endpoints}, this extension vanishes at
$0$ and $1$, and by Lemma \ref{lem:Pn-npoint-extension} it satisfies
the $n$-point Bellman inequality on all of $[0,1]$. Hence
$P^{(n)}\in \mathcal B(1,1,n)$. Applying Theorem \ref{mthm:1} with
$\alpha=\beta=1$ gives
$
P^{(n)}(|A|)\le \|S_1(\mathbbm 1_A)\|_1
$
for every measurable $A\subset[0,1)$. 
By Lemma \ref{lem:Pn-symmetry},
we get $
P^{(n)}(|A|)=P^{(n)}(|A|^*)$.
Therefore
\[
P^{(n)}(|A|^*)\le \|S_1(\mathbbm 1_A)\|_1,
\]
which proves item $(4)$ and completes the proof.
\end{proof}

\begin{remark}
There is another natural candidate obtained by replacing the digit-sum
statistic $\sigma_n$ with the number $s_n$ of nonzero base-$n$ digits.
This candidate is closer to the geometric interpretation of Hart's binary
formula. Numerical tests suggest that it may satisfy the same $n$-point
Bellman inequality, but we do not use this candidate in the present work as it would give a worse lower bound.
\end{remark}

\section{Proof of Theorem \ref{mthm4} and Corollary \ref{cor:P-dominates-omega}}

We first prove the comparison between $P^{(n)}$ and $\omega_n$.
By Theorem~\ref{Thm:digit-sum-Pn}, we have
\[
P^{(n)}(x)
=
\sum_{j=0}^{\infty}n^{-j}\eta_n(\{n^jx\}),
\]
where $\eta_n:[0,1]\to[0,\infty)$ is the continuous piecewise linear
function determined by
\[
\eta_n\left(\frac{\rho}{n}\right)
=
\frac{\rho(n-\rho)}{n(n-1)},
\qquad
\rho=0,1,\dots,n.
\]
On the other hand, by \eqref{eq:wn} (see \cite{LevEdgeIso}), we have 
\[
\omega_n(x)
=
\sum_{j=0}^{\infty}n^{-j}\varphi_n(\{n^jx\}),
\qquad
\varphi_n(t):=\|t\|_{1/n}
=
\min\{\operatorname{dist}(t,\mathbb Z),1/n\}.
\]
Thus it is enough to prove
\[
\eta_n(t)\ge \varphi_n(t),
\qquad 0\le t\le1.
\]

On $[0,1/n]$, both functions are equal to $t$. On
$[1-1/n,1]$, both functions are equal to $1-t$. Now consider a middle
interval
\[
\frac{r}{n}\le t\le \frac{r+1}{n},
\qquad
1\le r\le n-2.
\]
On this interval,
\[
\varphi_n(t)=\frac1n.
\]
Meanwhile $\eta_n$ is linear between
\[
\eta_n\left(\frac{r}{n}\right)
=
\frac{r(n-r)}{n(n-1)}
\quad 
\text{and}
\quad 
\eta_n\left(\frac{r+1}{n}\right)
=
\frac{(r+1)(n-r-1)}{n(n-1)}.
\]
For every $1\le \rho\le n-1$, one has
\[
\rho(n-\rho)\ge n-1.
\]
Therefore both endpoint values are at least $1/n$. Since $\eta_n$ is
linear on the interval, it follows that
\[
\eta_n(t)\ge \tfrac1n=\varphi_n(t).
\]
Thus
$
\eta_n(t)\ge \varphi_n(t),
\qquad 0\le t\le1$. Summing term by term gives
\[
P^{(n)}(x)\ge \omega_n(x),
\qquad 0\le x\le1.
\]

For $n=2$, one has $\eta_2=\varphi_2$. Hence
$
P^{(2)}=\omega_2$.
For $n=3$, the functions $\eta_3$ and $\varphi_3$ agree at the nodes
\[
0,\frac13,\frac23,1,
\]
and both are linear on the corresponding subintervals. Hence
$\eta_3=\varphi_3,$
and therefore
$
P^{(3)}=\omega_3$.

For $n\ge4$, the comparison is strict at some points. Indeed, at
$t=2/n$, one has
\[
\eta_n\left(\frac2n\right)
=
\frac{2(n-2)}{n(n-1)}
>
\frac1n
=
\varphi_n\left(\frac2n\right).
\]
Taking $x=2/n$, all terms with $j\ge1$ vanish, since
$
n^j\tfrac2n=2n^{j-1}\in\mathbb Z.
$
Therefore
\[
P^{(n)}\left(\frac2n\right)
>
\omega_n\left(\frac2n\right).
\]

We next show that equality holds along the sequence $x=n^{-k}$.
Fix $k\ge1$ and set
$
p_k:=n^{-k}.
$
For $0\le j\le k-1$, we have
\[
\{n^jp_k\}
=
n^{j-k}
\in
\left(0,\tfrac1n\right].
\]
Since
$
\eta_n(t)=\varphi_n(t)=t,
$ for $
0\le t\le\tfrac1n,
$
it follows that
\[
n^{-j}\eta_n(\{n^jp_k\})
=
n^{-j}\varphi_n(\{n^jp_k\})
=
n^{-j}n^{j-k}
=
n^{-k}
=
p_k.
\]
For $j\ge k$, the number
$
n^jp_k=n^{j-k}
$
is an integer, and hence
$
\{n^jp_k\}=0.
$
Thus all terms with $j\ge k$ vanish. Consequently,
\begin{align*}
P^{(n)}(p_k)
&=
\sum_{j=0}^{k-1}
n^{-j}\eta_n(n^{j-k})
=
\sum_{j=0}^{k-1}p_k
=
kp_k,
\end{align*}
and similarly,
\begin{align*}
\omega_n(p_k)
=
\sum_{j=0}^{k-1}
n^{-j}\varphi_n(n^{j-k})
=
\sum_{j=0}^{k-1}p_k
=
kp_k.
\end{align*}
Therefore
\[
P^{(n)}(n^{-k})
=
\omega_n(n^{-k})
=
kn^{-k}.
\]
By the symmetry of both functions, equality also holds at
$
x=1-n^{-k}.
$
This proves Corollary~\ref{cor:P-dominates-omega}.\\

We now prove Theorem~\ref{mthm4}. Let
$A\subset[0,1)$ be measurable, and set 
$x:=|A|$.
By Theorem~\ref{Thm:digit-sum-Pn},
\[
\|S_1(\mathbbm 1_A)\|_1
\ge
P^{(n)}(|A|).
\]
By Corollary~\ref{cor:P-dominates-omega},
\[
P^{(n)}(|A|)
\ge
\omega_n(|A|).
\]
Therefore
\[
\|S_1(\mathbbm 1_A)\|_1
\ge
\omega_n(|A|).
\]
By Lemma~\ref{thm:wn-properties}(c),
\[
\omega_n(|A|)
=
\omega_n(1-|A|)
=
\omega_n(x^*),
\qquad
x^*:=\min\{|A|,1-|A|\}.
\]
Hence
\[
\|S_1(\mathbbm 1_A)\|_1
\ge
\omega_n(x^*).
\]
For $0<x<1$, Lemma~\ref{thm:wn-properties}(d) gives
\[
\omega_n(x^*)
\asymp_n
x^*\log\left(\frac1{x^*}\right).
\]

It remains to prove that the lower bound is sharp up to a constant
along the sequence $x=n^{-k}$. For $k\ge1$, let
$
A_k:=[0,n^{-k}),
$ and  $
p_k:=|A_k|=n^{-k}.$
Take 
$
f:=\mathbbm 1_{A_k}.
$,
for $0\le m\le k$, define
\[
I_m:=[0,n^{-m}).
\]
Since $A_k\subset I_m$, the average of $f$ over $I_m$ is
\[
\tfrac{|A_k|}{|I_m|}
=
\tfrac{n^{-k}}{n^{-m}}
=
n^{m-k}.
\]
Therefore
\[
f_m
=
n^{m-k}\mathbbm 1_{I_m},
\qquad
0\le m\le k.
\]
At level $k$, we have
$
f_k=\mathbbm 1_{A_k}=f,$
and hence
$
d_m=0,
$ for $
m>k$.

Fix $1\le m\le k$. Since
\[
d_m=f_m-f_{m-1},
\]
we have
\[
d_m(x)
=
\begin{cases}
n^{m-k}-n^{m-1-k},
&
x\in I_m,
\\[1ex]
-n^{m-1-k},
&
x\in I_{m-1}\setminus I_m,
\\[1ex]
0,
&
x\notin I_{m-1}.
\end{cases}
\]
Since
\[
n^{m-k}-n^{m-1-k}
=
(n-1)n^{m-1-k},
\]
this becomes
\[
d_m(x)
=
\begin{cases}
(n-1)n^{m-1-k},
&
x\in I_m,
\\[1ex]
-n^{m-1-k},
&
x\in I_{m-1}\setminus I_m,
\\[1ex]
0,
&
x\notin I_{m-1}.
\end{cases}
\]
Therefore
\begin{align*}
\|d_m\|_1
&=
|I_m|(n-1)n^{m-1-k}
+
|I_{m-1}\setminus I_m|n^{m-1-k}
\\
&=
n^{-m}(n-1)n^{m-1-k}
+
\left(n^{-(m-1)}-n^{-m}\right)n^{m-1-k}.
\end{align*}
The first term equals
\[
n^{-m}(n-1)n^{m-1-k}
=
(n-1)n^{-k-1}.
\]
Moreover,
$
n^{-(m-1)}-n^{-m}
=
(n-1)n^{-m},
$
so the second term also equals
\[
(n-1)n^{-m}n^{m-1-k}
=
(n-1)n^{-k-1}.
\]
Hence
\[
\|d_m\|_1
=
2(n-1)n^{-k-1}
=
\frac{2(n-1)}n\,p_k.
\]

Since $d_m=0$ for $m>k$, we have
$
S_1(f)
=
\sum_{m=1}^k|d_m|.
$
Therefore
\begin{align*}
\|S_1(\mathbbm 1_{A_k})\|_1
=
\int_0^1
\sum_{m=1}^k|d_m(x)|\,dx
=
\sum_{m=1}^k
\int_0^1|d_m(x)|\,dx
=
\sum_{m=1}^k\|d_m\|_1
=
\sum_{m=1}^k
\frac{2(n-1)}n\,p_k
=
\frac{2(n-1)}n\,kp_k.
\end{align*}
Since
$
\omega_n(p_k)=kp_k,
$
we conclude that
\[
\|S_1(\mathbbm 1_{A_k})\|_1
=
\frac{2(n-1)}n\,
\omega_n(n^{-k}).
\]
Thus the constant in Theorem~\ref{mthm4} may be taken to be
$
c_n=\frac{2(n-1)}n.
$
Equivalently, since the lower bound applies to every measurable set of
measure $p_k$, we have
\[
\omega_n(p_k)
\le
V_n(p_k)
\le
\|S_1(\mathbbm 1_{A_k})\|_1
=
\frac{2(n-1)}n\,\omega_n(p_k).
\]
Finally, for 
$p_k=n^{-k}\longrightarrow0$
we have
\[
\omega_n(p_k)
=
kp_k
=
\frac1{\log n}\,
p_k\log\left(\frac1{p_k}\right).
\]
Therefore the logarithmic order of the lower bound is optimal along the
sequence $p_k=n^{-k}$.

For $n=2$, one has
$
\frac{2(n-1)}n=1,
$
so the lower bound is attained exactly by the intervals $A_k$.
This completes the proof of Theorem~\ref{mthm4}.

\section{Proof of Theorem \ref{thm:alpha}}

In this section we prove the sub-$L^1$ lower bound and its sharpness.
We first prove the lower bound for every measurable $A\subset[0,1)$, and then prove sharpness using $n$-adic intervals.
\[
\|S_1(\mathbbm 1_A)\|_\alpha\ge |A|^*,
\qquad
|A|^*:=\min\{|A|,1-|A|\}.
\]
Moreover, this estimate is sharp up to a constant depending only on
$\alpha$ and $n$.

\begin{lemma}
\label{lem:xstar-admissible-alpha}
Define
\[
b(x):=x^*=\min\{x,1-x\},
\qquad 0\le x\le1.
\]
It holds $b\in\mathcal B(\alpha,1,n)$. That is $b(0)=b(1)=0$, and for every $x_1,\dots,x_n\in[0,1]$, one has 
\begin{equation}\label{eq:xstar-alpha-npoint}
b(\bar x)^\alpha
\le
\frac1n\sum_{i=1}^n
\bigl(b(x_i)+|x_i-\bar x|\bigr)^\alpha,\qquad 
\bar x=\frac1n\sum_{i=1}^n x_i.
\end{equation}
\end{lemma}

\begin{proof}
The endpoint conditions are immediate.
We use the elementary interpretation
\[
b(x)=\operatorname{dist}(x,\{0,1\}).
\]
By the triangle inequality for distance to a set, for each $i=1,\dots,n$,
\[
b(\bar x)
=
\operatorname{dist}(\bar x,\{0,1\})
\le
\operatorname{dist}(x_i,\{0,1\})+|x_i-\bar x|
=
b(x_i)+|x_i-\bar x|.
\]
Since $\alpha <1$, the map $t\mapsto t^\alpha$ is increasing on $[0,\infty)$, we get
\[
b(\bar x)^\alpha
\le
\bigl(b(x_i)+|x_i-\bar x|\bigr)^\alpha
\]
for every $i$. Averaging over $i$ gives
\[
b(\bar x)^\alpha
\le
\frac1n\sum_{i=1}^n
\bigl(b(x_i)+|x_i-\bar x|\bigr)^\alpha.
\]
This proves \eqref{eq:xstar-alpha-npoint}. Thus $b\in\mathcal B(\alpha,1,n)$.
\end{proof}

\begin{corollary}[The lower bound]
\label{cor:alpha-lower-bound}
For every measurable $A\subset[0,1)$,
\[
\|S_1(\mathbbm 1_A)\|_\alpha
\ge
|A|^*.
\]
\end{corollary}

\begin{proof}
By Lemma \ref{lem:xstar-admissible-alpha}, the function $b(x)=x^*$
belongs to the admissible Bellman class $\mathcal B(\alpha,1,n)$.
Hence, by Theorem \ref{mthm:1} with $\beta=1$, we obtain
$
\|S_1(\mathbbm 1_A)\|_\alpha
\ge
|A|^*.
$
This proves the lower bound.
\end{proof}

It remains to prove sharpness. We use the same extremizing sequence as in the
dyadic case, namely the leftmost $n$-adic intervals.

\begin{lemma}[Sharpness along $n$-adic intervals]
\label{lem:alpha-sharpness-nadic}
Let
\[
A_k:=\left[0,\tfrac1{n^k}\right),
\qquad
p_k:=|A_k|=n^{-k}.
\]
Then there exists a constant $C_{\alpha,n}<\infty$ such that
\[
\|S_1(\mathbbm 1_{A_k})\|_\alpha
\le
C_{\alpha,n}\,p_k
\qquad
\text{for every }k\ge1.
\]
Consequently, the lower bound
$
\|S_1(\mathbbm 1_A)\|_\alpha\ge |A|^*
$
is sharp up to a constant factor as $|A|\to0$.
\end{lemma}

\begin{proof}
Fix $k\ge1$, and set
$
p:=p_k=n^{-k},
$ and  $
A_k=\left[0,n^{-k}\right)$.
Let $f=\mathbbm 1_{A_k}$, and let $f_m=\mathbb E(f\mid\mathcal N_m)$
be the martingale associated with the regular $n$-adic filtration.
For $0\le m\le k$, the only level-$m$ interval that intersects $A_k$
is the leftmost interval
\[
I_m=\left[0,n^{-m}\right).
\]
On $I_m$, one has
$f_m=n^{m-k},$
and outside $I_m$, one has $f_m=0$. In particular, on $A_k$,
\[
f_0=p,\qquad f_1=n^{1-k},\qquad \dots,\qquad f_k=1.
\]
Thus for $x\in A_k$,
\[
S_1(f)(x)
=
\sum_{m=1}^k
\left(n^{m-k}-n^{m-1-k}\right)
=
1-p.
\]
Therefore the contribution of $A_k$ to the integral of $S_1$ can be seen as
\[
\int_{A_k}S_1(f)^\alpha\,dx
=
p(1-p)^\alpha
\le p.
\]

Now decompose the complement of $A_k$ according to the first generation at
which the point leaves the leftmost branch. For $1\le \ell\le k$, define
\[
R_\ell
:=
\left[n^{-\ell},\,n^{-(\ell-1)}\right),\qquad
|R_\ell|
=
(n-1)n^{-\ell}.
\]
If $x\in R_\ell$, then the martingale follows the leftmost branch up to
level $\ell-1$, and becomes zero from level $\ell$ onward. Hence
\[
S_1(f)(x)
=
\sum_{m=1}^{\ell-1}
\left(n^{m-k}-n^{m-1-k}\right)
+
n^{\ell-1-k}.
\]
The first sum telescopes:
\[
\sum_{m=1}^{\ell-1}
\left(n^{m-k}-n^{m-1-k}\right)
=
n^{\ell-1-k}-p.
\]
Therefore, for $x\in R_\ell$,
\[
S_1(f)(x)
=
2n^{\ell-1-k}-p
=
p\left(2n^{\ell-1}-1\right).
\]
Consequently,
\begin{align*}
\int_{[0,1]\setminus A_k}S_1(f)^\alpha\,dx
&=
\sum_{\ell=1}^{k}
|R_\ell|
\left(p(2n^{\ell-1}-1)\right)^\alpha
\\
&=
p^\alpha
\sum_{\ell=1}^{k}
(n-1)n^{-\ell}
(2n^{\ell-1}-1)^\alpha
\\
&\le
p^\alpha
\sum_{\ell=1}^{k}
(n-1)n^{-\ell}
(2n^{\ell-1})^\alpha
=
p^\alpha
(n-1)2^\alpha n^{-\alpha}
\sum_{\ell=1}^{k}
n^{\ell(\alpha-1)}.
\end{align*}
Since $\alpha<1$,
\[
\sum_{\ell=1}^{k}n^{\ell(\alpha-1)}
\le
\sum_{\ell=1}^{\infty}n^{\ell(\alpha-1)}
=
\frac{n^{\alpha-1}}{1-n^{\alpha-1}}.
\]
Hence
\[
\int_{[0,1]\setminus A_k}S_1(f)^\alpha\,dx
\le
C'_{\alpha,n}p^\alpha,
\]
where
\[
C'_{\alpha,n}
:=
(n-1)2^\alpha n^{-\alpha}
\frac{n^{\alpha-1}}{1-n^{\alpha-1}}
=
\frac{(n-1)2^\alpha}{n(1-n^{\alpha-1})}.
\]

Combining the contribution from $A_k$ and from its complement, and using
$p\le p^\alpha$ for $0<p\le1$, we get
\[
\int_0^1S_1(f)^\alpha\,dx
\le
p+C'_{\alpha,n}p^\alpha
\le
(1+C'_{\alpha,n})p^\alpha.
\]
Therefore
\[
\|S_1(\mathbbm 1_{A_k})\|_\alpha
=
\left(\int_0^1S_1(f)^\alpha\,dx\right)^{1/\alpha}
\le
(1+C'_{\alpha,n})^{1/\alpha}p.
\]
Thus the desired estimate holds with
$
C_{\alpha,n}:=(1+C'_{\alpha,n})^{1/\alpha}.
$
Since $p_k=n^{-k}\to0$, the lower bound is sharp up to a constant factor.
\end{proof}

\begin{proof}[Proof of Theorem \ref{thm:alpha}]
The lower bound follows directly from  Corollary \ref{cor:alpha-lower-bound}. The sharpness statement follows
from Lemma \ref{lem:alpha-sharpness-nadic}, applied to the sequence
$
A_k=\left[0,n^{-k}\right).
$
This completes the proof.
\end{proof}

\appendix

\section{Exact verification of the local debt inequality}
\label{app:local-debt-verification}

This appendix proves Lemma \ref{lem:local-debt}. More explicitly, $\Lambda$, $\mad$, and $\mathcal L_a$ are obtained
from rational affine functions using finitely many absolute values and
piecewise-linear branches, and $\Gamma$ is obtained from the finitely many
functions $-\mathcal L_c$ by taking their maximum together with $0$.
Thus each inequality below is a finite collection of rational affine
inequalities on rational polyhedral cells.

No floating point arithmetic and no numerical sampling are used. The code
below uses exact symbolic quantifier elimination over the ordered field of the real
numbers with rational coefficients. Thus
the output is not evidence from testing many points; it is an exact
verification that no real counterexample exists.

Recall the definitions from Section \ref{sec:proof-mthm2}. For
$a,b\in\{0,1,2\}^3$ and $r\in[0,1]^3$, define $\mathcal L_a(r)$ by
\eqref{eq:local-remainder}, and define
\[
\Gamma(r)
=
\max\Bigl\{0,\,-\mathcal L_c(r):c\in\{0,1,2\}^3\Bigr\}.
\]
For fixed $a$, put
\[
Y_a(r):=\Gamma(r)+\mathcal L_a(r).
\]
Since $\Gamma(r)\ge-\mathcal L_a(r)$, one has $Y_a(r)\ge0$. Hence
\[
Y_a(r)\ge 3\Gamma\left(\frac{a+r}{3}\right)
\]
is equivalent to
\[
Y_a(r)
+
3\mathcal L_b\left(\frac{a+r}{3}\right)
\ge0
\]
for every $b\in\{0,1,2\}^3$. Therefore it is enough to verify the
$27\cdot27=729$ inequalities
\begin{equation}\label{eq:appendix-729-ineq}
\Gamma(r)+\mathcal L_a(r)
+
3\mathcal L_b\left(\frac{a+r}{3}\right)
\ge0,
\qquad
a,b\in\{0,1,2\}^3,
\qquad
r\in[0,1]^3.
\end{equation}
The endpoint statement in Lemma \ref{lem:local-debt} is the identity
\[
\Gamma(r)=0,
\qquad r\in\{0,1\}^3.
\]

The function \texttt{red} implements the reduction
\[
s=K-3\min\{\lfloor K/3\rfloor,2\},
\qquad 0\le K\le9.
\]
We use half-open intervals at $3$ and $6$ only to make the endpoint
convention explicit. Since $\Lambda(0)=\Lambda(3)=0$, the value of
$\Lambda(s)$ would be unchanged if $s=3$ were used instead of $s=0$
at these two endpoints, but the convention below matches the definition
used in the proof.

The following standalone Wolfram Language script verifies both statements exactly.
\begin{verbatim}
ClearAll["Global`*"];

digits = Tuples[{0, 1, 2}, 3];

lam[u_] := Piecewise[{
{4 u/3, 0 <= u <= 1},
{2 - 2 u/3, 1 <= u <= 3/2},
{2 u/3, 3/2 <= u <= 2},
{4 (3 - u)/3, 2 <= u <= 3}
}];

red[u_] := Piecewise[{
{u, 0 <= u < 3},
{u - 3, 3 <= u < 6},
{u - 6, 6 <= u <= 9}
}];

mad[v_] := Total[Abs[# - Total[v]/3] & /@ v];

loc[a_, r_] := Module[{R, K, y},
R = Total[r];
K = Total[a] + R;
y = a + r;
lam[R]
+ Total[lam /@ y]
+ mad[y]
- mad[r]
- lam[red[K]]
- 3 lam[K/3]
];

gamma[r_] := Max @@ Join[{0}, (-loc[#, r] & /@ digits)];

bad[a_, b_] := Resolve[
Exists[{r1, r2, r3},
0 <= r1 <= 1 &&
0 <= r2 <= 1 &&
0 <= r3 <= 1 &&
gamma[{r1, r2, r3}]
+ loc[a, {r1, r2, r3}]
+ 3 loc[b, (a + {r1, r2, r3})/3] < 0
],
Reals
];

localDebtTest =
Tally[
Flatten[
 Table[
    bad[a, b],
    {a, digits},
    {b, digits}
 ]
]
];

endpointTest =
Table[
gamma[v],
{v, Tuples[{0, 1}, 3]}
];

{localDebtTest, endpointTest}
\end{verbatim}

The output is

\begin{verbatim}
{{{False, 729}}, {0, 0, 0, 0, 0, 0, 0, 0}}
\end{verbatim}

All constants in the code are rational; no machine-real numbers occur.\\

The first component says that, for each of the $729$ pairs $(a,b)$, the
existential formula searching for a counterexample to
\eqref{eq:appendix-729-ineq} is false. In other words, for every pair
$a,b\in\{0,1,2\}^3$, there is no real point
$r=(r_1,r_2,r_3)\in[0,1]^3$ at which the left-hand side of
\eqref{eq:appendix-729-ineq} is negative. Hence all $729$ inequalities
hold on the entire cube $[0,1]^3$.

The second component says that $\Gamma(r)=0$ at all eight vertices
$r\in\{0,1\}^3$. Therefore Lemma \ref{lem:local-debt} follows.\\

Let us stress that the command \texttt{bad[a,b]} is not checking a finite
list of sample points. It asks whether there exists a real point
$r=(r_1,r_2,r_3)\in[0,1]^3$ at which the left-hand side of
\eqref{eq:appendix-729-ineq} is negative. Since all functions appearing in
this formula are built from rational affine functions, absolute values,
finite maxima, and finite piecewise affine branches, this is a first-order
statement over the ordered field of real numbers with rational
coefficients. The command \texttt{Resolve[..., Reals]} decides this
statement exactly.

\bibliographystyle{plain}
\bibliography{references}

\end{document}